\documentclass[12pt,a4paper]{amsart}
\usepackage{amsmath}
\usepackage{amsfonts,amssymb,amscd,amsthm,comment,cite,float,latexsym,psfrag,psfig,url}
\usepackage{lipsum}
\usepackage[htt]{hyphenat}
\usepackage[dvips]{graphicx}
\usepackage[thinlines]{easytable}
\usepackage{array}
\usepackage[shortlabels]{enumitem}
\usepackage{mathtools}
\usepackage{hyperref}
\usepackage{xparse}
\usepackage{comment}
\usepackage{verbatimbox}
\usepackage{lipsum}
\usepackage{xcolor}
\usepackage{setspace}

\newtheorem{thm}{Theorem}[section]
\newtheorem*{thmA}{Theorem A}
\newtheorem*{thmB}{Theorem B}  
\newtheorem*{thm-non}{Theorem}
\newtheorem*{exthm}{Theorem: Existence of a renormalization fixed point}

\newtheorem{lem}[thm]{Lemma}

\newtheorem{conj}[thm]{Conjecture}
\theoremstyle{remark}
\newtheorem{rem}{Remark}[section]
\theoremstyle{remark}
\newtheorem{alg}{Algorithm}[section]

\newtheorem{problem}[thm]{Problem}

\theoremstyle{definition}

\numberwithin{equation}{section}
\numberwithin{figure}{section}

\font\nt=cmr7

\def\note#1
{\marginpar
{\nt $\leftarrow$
\par
\hfuzz=20pt \hbadness=9000 \hyphenpenalty=-100 \exhyphenpenalty=-100
\pretolerance=-1 \tolerance=9999 \doublehyphendemerits=-100000
\finalhyphendemerits=-100000 \baselineskip=6pt
#1}\hfuzz=1pt}

\newcommand{\co}{\circ}

\newcommand{\HD}{\operatorname{HD}}

\newcommand{\id}{\operatorname{id}}

\newcommand{\transverse}
 {\kern .7em\makebox[0pt][c]{\raisebox{.2ex}{$|$}}\kern -.6em\cap}

\newcommand{\tangent}
 {\kern .7em\makebox[0pt][c]{\raisebox{.77ex}{$--$}}\kern -.6em\cap}

\newcommand{\eps}{{\varepsilon}}

\newcommand{\cO}{{\mathcal O}}
\newcommand{\cR}{{\mathcal R}}

\newcommand{\cC}{{\mathcal C}}

\newcommand{\cD}{{\mathcal D}}
\newcommand{\cT}{{\mathcal T}}

\newcommand{\cB}{{\mathcal B}}
\newcommand{\cE}{{\mathcal E}}
\newcommand{\cS}{{\mathcal S}}
\newcommand{\cP}{{\mathcal P}}
\newcommand{\cA}{{\mathcal A}}
\newcommand{\cZ}{{\mathcal Z}}

\newcommand{\cN}{{\mathcal N}}

\newcommand{\C}{{\mathbb C}}

\newcommand{\D}{{\mathbb D}}

\newcommand{\R}{{\mathbb R}}

\newcommand{\I}{{\mathbb I}}

\def\BPhi{{\boldsymbol{\BPhi}}}
\def\B0{{\mathbf{0}}}

\includecomment{proof-comment}
\includecomment{proof-claim-comment}

\newtheorem{Th}{Theorem}
\newtheorem{Prop}[Th]{Proposition}
\newtheorem{Lm}[Th]{Lemma}
\newtheorem{Co}[Th]{Corollary}

\theoremstyle{definition}
\newtheorem{Def}[Th]{Definition}

\newtheorem{Rem}{Remark}

\def\id{\text{Id}}

\newcommand{\ie}{\emph{i.e.}\ }

\newcommand{\hyper}{\mathrm{hyp}}

\catcode`\@=12

\def\Empty{}
\newcommand\oplabel[1]{
  \def\OpArg{#1} \ifx \OpArg\Empty {} \else
  	\label{#1}
  \fi}
		
\newcommand{\comm}[1]{}

\NewDocumentCommand{\ceil}{s O{} m}{%
  \IfBooleanTF{#1} 
    {\left\lceil#3\right\rceil} 
    {#2\lceil#3#2\rceil} 
}

\begin{document}


\bigskip\bigskip

\title[Wild attractors for Fibonacci maps]{Wild attractors for Fibonacci maps\\}

\author{A. Dudko}
\address{Institute of Mathematics\\
  Polish Academy of Sciences\\
  ul. \'Sniadeckich 8\\
  00-656 Warsaw\\
  Poland \\ and B. Verkin Institute for Low Temperature Physics and Engineering \\ 47 Nauky Ave. \\ Kharkiv 61103\\ Ukraine
}
\email[A.~Dudko]{adudko@impan.pl}

\author{D. Gaidashev}


\address{Department of Mathematics\\
  Uppsala University\\
  Box 480 \\
  751 06   Uppsala\\
  Sweden}
\email[D. ~Gaidashev]{gaidash@math.uu.se}

\date{\today}

\begin{abstract}
  Existence of wild attractors - attractors whose basin has a positive Lebesgue measure but is not a residual set - has been one of central themes in one-dimensional dynamics.

It has been demonstrated by H. Bruin {\it et al.} in \cite{BKNS} that Fibonacci maps with a sufficiently flat critical point admit a wild attractor.

We propose a constructive trichotomy that describes possible scenarios for the Lebesgue measure of the Fibonacci attractor based on a computable criterion. We use this criterion, together with a computer-assisted proof of existence of a Fibonacci renormalization $2$-cycle for non-integer critical degrees, to demonstrate that Fibonacci maps do not have a wild attractor when the degree of the critical point is $d=3.8$ (and, conjecturally, for $2< d \le 3.8$), and do admit it when $d=5.1$ (and, conjecturally, for $d \ge 5.1$).
\end{abstract}

\maketitle

\setcounter{tocdepth}{1}
\tableofcontents

\section{Introduction}

\subsection{Wild attractors} One of the key objects in studying dynamical systems is the notion of an \emph{attractor}. Informally speaking, an attractor of a dynamical system is the set attracting "large" number of orbits of this system. More specifically, let $f:X\to X$ be a map, where $X$ is a topological space. For a point $x\in X$ let $\omega(x)$ be the $\omega$-limit set of $x$. For a subset $\cC \subset X$ let $B(\cC)=\{x\in X:\omega(x)\subset \cC\}$ be the basin of attraction of $\cC$. Let $\mu$ be a measure on $X$. A subset $\cC \subset X$ is called a
\begin{itemize}
\item{} \emph{topological attractor} if $B(\cC)$ is a residual set and is minimal with this property,
\item{} \emph{metric attractor} if $B(\cC)$ has positive Lebesgue measure and is minimal with this property.
\end{itemize}
One of possible exotic situations is when metric and topological attractors of a system do not coincide. A metric attractor which is not a topological attractor is called a \emph{wild attractor}. Observe that for a smooth map $f$ with critical points a natural candidate for a metric attractor is the closure of the postcritical set $\cP(f)$ of $f$ (closure of the union of forward orbits of critical points).

The first examples of polynomial maps on a segment having a wild attractor were shown to exist among Fibonacci polynomial maps in \cite{BKNS}. Recall that for any $d\geqslant 2, d\in\mathbb R,$ there exists a unique $c_d\in\mathbb R$ such that the map $P_d(z)=|z|^d+c_d$ has Fibonacci combinatorics of the postcritical set {(see \cite{BKNS} or \cite{LM} for precise definitions)}. Notice that when $d$ is a positive even integer $P_d(z)$ is a polynomial. A Fibonacci map has a unique critical point at the origin. Observe that $P_d$ is non-renormalizable in the classical sense and does not have non-repelling periodic orbits. According to classification of topological attractors \cite{Guck-1979}, it has a unique topological attractor, which is a cycle of segments. Thus, $\overline{\mathcal P(P_d)}=\omega(0)$ cannot be a topological attractor.

\begin{Th}[Bruin-Keller-Novicki-van Strien] There exists $d_0>0$ such that for any $d\geqslant d_0$ the Fibonacci map $|z|^d+c_d$ has a wild attractor.\end{Th}
\noindent However, the estimate for $d_0$ one can obtain from \cite{BKNS} is very large. Results of the present paper {suggest} that already $P_{5.1}(z)$ has a wild attractor, however $P_{3.8}(z)$ does not have one.

{We notice that other classes of maps with Fibonacci type combinatorics were studied, for instance, in \cite{Bruin-98} and \cite{BruinTodd-15}}.

\subsection{Renormalization for Fibonacci maps}
{The Fibonacci polynomials $P_d$ are not renormalizable in the usual sense. However, one can study the dynamics of $P_d$ using an appropriate generalized renormalization introduced in \cite{LM}.}

M. Lyubich and J. Milnor {study} in \cite{LM} a class of $C^2$ maps, which differ from a usual dynamical system in that their domain, ${\rm Dom}(F)$ is a proper subset of the range, ${\rm R}(F)=[-1,1]$. Specifically, ${\rm Dom}(F)=J_1 \cup T_1$ with $J_1=[a,b]$ and $T_1=[\alpha,\beta]$ where $-1 <a <b<\alpha < \beta<1$. This class of maps satisfies the following conditions:

\begin{itemize}
\item[$(i)$] $F |_{J_1}$ is a diffeomorphism from $J_1$ onto $[-1,1]$, which may be either orientation preserving, or orientation reversing;
\item[$(ii)$] $F |_{T_1}$ is a unimodal map from  $T_1$ onto $[-1,1]$, with a non-degenerate minimum $x_0=0$ and with $F(\partial aT_1)=1$.
\end{itemize}

The class $\cA$ of such maps is union of two disjoint components $\cA^+$ and $\cA^-$, for which $F_{J_1}$ is orientation preserving or reversing, respectively.

Under the condition $F(x_0) \in J_1$ and $F^2(x_0) \in T_1$, there exists an interval $T_2 \ni x_0$, mapped unimodally into $T_1$ by $F^2$ with both endpoints of $T_2$ mapping to a single endpoint of $T_1$. Also, there is an interval $J_2 \subset T_1$, $J_2 \cap T_2 = \emptyset$, such that $F |_{J_2}$ is a diffeomorphism onto $T_1$. There always two choices of such $J_2$, and one chooses $J_2$ to the left of $T_2$ if $F|_{J_1}$ is orinetation preserving, to the right of $T_2$ otherwise.

\begin{figure}[t]
\begin{center}
 \includegraphics[angle=-90,width=0.8 \textwidth]{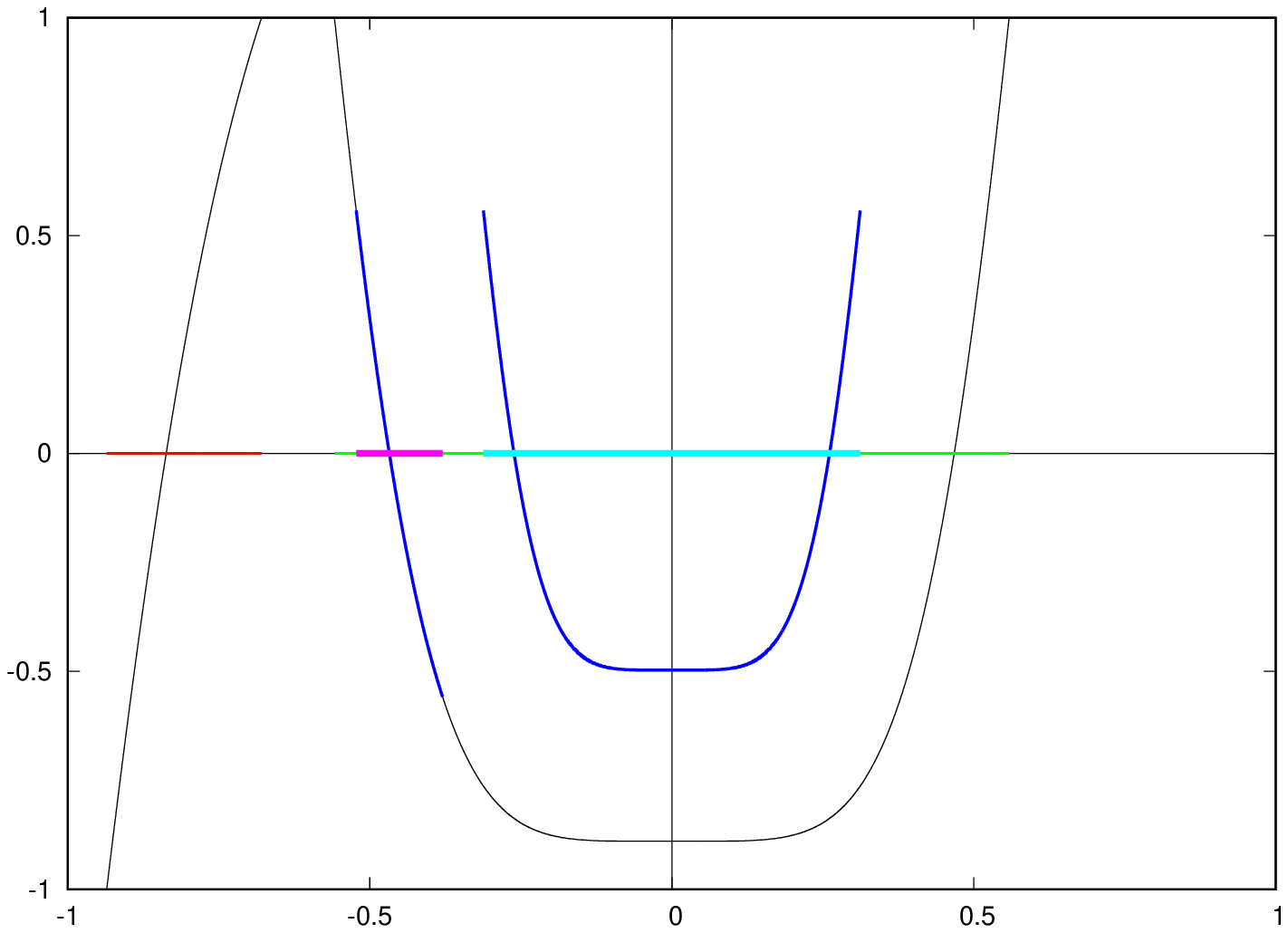}
\end{center}
\caption{A renormalizable Fibonacci map of degree $d=5.0$. The original map is defined over the intervals \textcolor{red}{$J_1$} and  \textcolor{green}{$T_1$}. Its prerenormalization, in blue, is defined over the intervals \textcolor{magenta}{$J_2$}  $\subset$ \textcolor{green}{$T_1$} and  \textcolor{cyan}{$T_2$} $\subset$ \textcolor{green}{$T_1$}. }
\end{figure}

Such map from $J_2 \cup T_2$ onto $T_1$ is referred to as the prerenormalization of $F$, denoted $P F$ or $F_1$. $P F$ rescaled by the affine transformation $s_{T_1}$ that map $T_1$ onto $[-1,1]$, is called the renormalization of $F$, denoted $R F$. Out of two possible choices of $s_{T_1}$ (orientation preserving and orientation reversing), we always choose the one for which $x_0$ is a minimum of $R F$. If $F$ is $n$-times renormalizable, its $n$-th prerenormalization will be denoted $F_n$.

The operator $R$ interchanges the components $\cA^\pm$.

M. Lyubich and J. Milnor have demonstrated in  \cite{LM} that a map $F \in \cA$ is infinitely renormalizable iff it is a Fibonacci map. Furthermore, a they also show that the restriction of a Fibonacci polynomial $P_2$ to the union of segments $[P_2^5(0),P_2^2(0)]$ and $[P_2(0),P_2^4(0)]$ can be extended to a smooth map of class $\mathcal A^-$. Similar is true for the Fibonacci polynomial $P_d$ with $d\geqslant 2$.

The authors of \cite{LM} also consider analytic Fibonacci maps in the Epstein class $\cE$, that is $F \in \cA$, such that $\left(F|_{T_1}\right)^{-1}:[0,1] \mapsto [0,1]$ extends analytically to a map from $\C_{[0,1]}$ to itself, and  $\left(F|_{J_1}\right)^{-1}:[-1,1] \mapsto J_1$ extends analytically to a map from $\C_{[-1,1]}$ to $\C_{J_1}$, where, given a closed interval $I$,
$$\C_I=(\C \setminus \R) \cup I.$$
For such maps, it is shown that their renormalizations eventually become polynomial-like of type $(d,1)$, that is, for sufficiently large $n$, the domain of the analytic extention of $R^n F$ contains two topolgical disks $U_{T_1}$ and $U_{J_1}$, $U_{T_1} \subset V$ and $U_{J_1} \subset V$, such that $F$ maps  $U_{T_1} \cup U_{J_1}$ onto $V$ and $F|_{U_{T_1}}$ is a branched covering of degree $d$, while  $F|_{U_{J_1}}$ is a covering of degree $1$.

The renormalization theory for the Fibonacci maps developped in \cite{LM} has been eventualy used by the authors to show that for $d \le 2$, the geometry of the closure of the critical orbit $\bar \cO$ (a Cantor set) is unbounded, specifically, if $I_n$ is an interval in the $n$-th level cover of the Cantor set $\bar \cO$, and $I_{n+1}' \subset I_n$ is that in the $n+1$-st, then $|I'_{n+1}|/|I_n| \sim a/2^{n/3}$ for some constant $a$, and, consequently, the Hausdorff dimension of $\bar \cO$ is zero. Any two Fibonacci maps with a non-degenerate critical point ($d \le 2$) with the same parameter $a$ are smoothly conjugate on $\bar \cO$.

{We would like to emphasize that in the present paper we study smooth Fibonacci maps on unions of segments. In several papers, complex analytic generalized polynomial-like Fibonacci maps were studied. In particular,} S. van Strien and T. Nowicki have proved in \cite{nowicki1994polynomial} that the renormalization operator for generalized polynomial-like Fibonacci maps, with $d$ an even natural larger than $2$, has a unique periodic orbit of period $2$. {Its geometric properties were studied in \cite{Buff-00}.} This renormalization period $2$ orbit has been shown to be hyperbolic by D. Smania in \cite{Sma}, it is also shown in \cite{Sma} that sufficiently high renormalizations of the infinitely renormalizable Fibonacci map belong to the local stable renormalization manifold at the period $2$ orbit.

{In the context of smooth Fibonacci maps, it is believed that the following is true:
\begin{conj}\label{conj_existence} For every $d>2$ there exists a unique period 2 periodic  point of the generalized renormalization with the Fibonacci combinatorics and critical point of degree $d$.
\end{conj}}
\noindent Although a demonstration of existence, uniqueness and hyperbolicity of a Fibonacci cycle for all $d>2$ is not an objective of this paper, we do obtain existence and hyperbolicity for several instances of $d$, specifically, $d=3.8$ and $d=5.1$, used in our proofs, using computer-assisted methods. { This can be viewed as a step towards proving Conjecture \ref{conj_existence}.} An analytic proof of existence and hyperbolicity of a Fibonacci cycle for general $d$'s deserves, however, a separate study.

\subsection{Avila-Lyubich trichotomy for Fibonacci maps} A natural question is how large is the set attracted by an attractor $A$. In \cite{Avila_Lyubich_08} and \cite{AvilaLyubich-Area-22} A. Avila and M. Lyubich developed methods for studying questions about Hausdorff dimension and Lebesgue measure of attractors of infinitely-renormalizable maps. Though they focused on Feigebnaum maps, they obtained a series of results for Fibonacci maps.

In Section 10 of \cite{Avila_Lyubich_08}, they showed for the basin of attraction $J^r(f)=\mathcal B(\overline{\mathcal P(f)})$ of the postcriticall set of a Fibonacci map $f$ with critical point of degree $d\in \mathbb R$:
\begin{Th}[Avila-Lyubich]\label{thm_Aliva-Lyubich} Assume that Conjecture \ref{conj_existence} is true. Then the following cases for $J^r(f)$ occur:
\begin{itemize}
\item[$(1)$] (Lean case) For an open set of $d$'s containing $(2,2+\epsilon)$ for some $\epsilon>0$ one has $\HD_\hyper(J^r(f))<1$.
\item[$2)$] (Balanced case) For a compact non-empty set of $d$'s, $\HD_\hyper(J^r(f))=\HD(J^r(f))=1$ but the Lebesgue measure of $J^r(f)$ is zero.
\item[$3)$] (Black hole case) For an open set of $d$'s containing $(K,\infty)$ for some $K>2$, $\HD_\hyper(J^r(f))<1$ but $J^r(f)$ has positive Lebesgue measure.
\end{itemize}
\end{Th}
See \cite{Lyubich-ICM-95} for the early history of the question.

\subsection{Acknowledgement}
The first author acknowledges the funding by Long-term program of support of the Ukrainian research teams at the Polish Academy of Sciences carried out in collaboration with the U.S. National Academy of Sciences with the financial support of external partners. The first author was partially supported by National Science
Centre, Poland, Grant OPUS21 "Holomorphic dynamics, fractals, thermodynamic formalism",
2021/41/B/ST1/00461.

The authors are grateful to Mikhail Lyubich for motivating the study of Fibonacci attractors and useful comments.

\section{Trichotomy for attractors of Fibonacci maps: Statement of the main result}
Throughout this section let $d>2$ be a fixed number. Let $\{F,G\}=\{(f,g),(f,-g)\}$ be the period two periodic cycle of Fibonacci renormalization with a critical point of degree $d$. Here $f:T_1=[-\lambda,\lambda]\to T_0=[-1,1]$ is an even unimodal map with $f(t)=1$, and $g:J_1=[i,j] \to T_0$ is an increasing homeomorphism, $i<f(0)<j$.

In what follows, we will introduce most of our definitions and results for the periodic point $F$. Clearly, all statements have corresponding counterparts for the periodic point $G$.

For $n \geqslant 0$, let $T_{n+1}$ and $J_{n+1}$ be the domains of the n-th prerenormalization $F_n=(f_n,g_n)$ of $F$, so that
\begin{align*}
  T_{n+1}&=\lambda^n T_1,  \\
  J_{n+1}&=\pm\lambda^n J_1,\\
  f_n(z)& =\pm\lambda^n f(z/\lambda^n), \\
  g_n(z)& =\pm\lambda^n g(\pm z/\lambda^n),
\end{align*}
where $\lambda$ is the scaling ratio (depending on $d$). Denote by
\begin{equation*}
   \cP =\{F^k(0):k\in\mathbb Z_+\}
\end{equation*}
the postcritical set of $F$, and
\begin{equation*}
   \cC =\overline{\cP},
\end{equation*}
the attractor of $F$. Introduce the sets
\begin{align}
\label{Xn}  X_n&=\{x\in T_1:F^k(x)\in T_n\;\;\text{for some}\;\;k \ge 0\},\\
\label{Yn}  Y_n&=\{y\in T_n:F^k(x)\notin T_n\;\;\text{for all}\;\; k \ge 1 \},\\
\label{Zn}  Z_n&=\left(\bigcup\limits_{k\geqslant 0}F_{n-1}^{-k}(Y_n)\right)\cap T_n
\end{align} and the quantities
$$\eta_n=|X_n|/|T_1|,\;\;\zeta_n=|Z_n|/|T_n|, $$ where $|\cdot|$ stands for the Lebesgue measure of a set. 
Denote by $B( \cC)$ the Basin of attraction of $ \cC$.
One of the main results of this paper is a trichotomy for wild attractors of Fibonacci maps,  similar to Avila-Lyubich's trichotomy for Feigenbaum Julia sets:

\begin{thmA}{\it (\underline{Constructive trichotomy for Fibonacci maps}.)}
There exists a constructive constant $C>1$ (depending on $d$), such that one of the following holds:
\begin{itemize}
\item[$(1)$] There exists $n\in\mathbb N$ such that $\eta_n/\zeta_n<C^{-1}$. Then $\eta_n\to 0$ exponentially fast, $\inf\zeta_n>0$, and $F$ does not have a wild attractor ($B( \cC)$ has zero Lebesgue measure).
\item[$(2)$] For all $n$ one has $C^{-1}\leqslant \eta_n/\zeta_n\leqslant C$. Then $\eta_n\to 0$, $\zeta_n\to 0$, and $F$ does not have a wild attractor ($B( \cC)$ has zero Lebesgue measure).
\item[$(3)$] There exists $n\in\mathbb N$ such that $\eta_n/\zeta_n>C$. Then $\inf\eta_n>0$, $\zeta_n\to 0$, and $F$ has a wild attractor ($B( \cC)$ has positive Lebesgue measure).
  \end{itemize}
\end{thmA}

Avila-Lyubich trichotomy suggests that in the third case $B( \cC)$ has a hyperbolic dimension strictly less than $1$, and in the second case $B( \cC)$ have Hausdorff dimension equal to one but zero measure. Notice that the main conclusion we make in the first and the second case is the same: $B( \cC)$ has zero Lebesgue measure. However, there is a principal difference between them from a computational point of view. Indeed, the first case can be verified in finite time, in contrast to the second. {The strategy of the proof of Theorem A follows \cite{DS-20} and \cite{Dudko_20}, where an analogue of part (1) of the trichotomy was proven for some complex analytic Feigenbaum maps.}

In the second part of the paper we use the above trichotomy together with rigorous computer estimate on $\zeta_n$ and $\eta_n$  to prove the following result.

\begin{thmB}{\it (\underline{Wild attractors for Fibonacci maps}.)}
\begin{itemize}
\item[$(1)$] The Fibonacci renormalization cycle does not admit a wild attractor for $d=3.8$.
\item[$(2)$] The Fibonacci renormalization cycle admits a wild attractor for $d=5.1$.
\end{itemize}
\end{thmB}

We remark that at present the result about the existence of the wild attractor concerns only the renormalization period $2$ cycle. A further renormalization study, specifically, a proof of renormalization hyperbolicity and rigidity of infinitely renormalizable maps, is required to extend these results to a larger set of infinitely renormalizable Fibonacci maps.

\section{Trichotomy for attractors of Fibonacci maps: The construction}
\subsection{Copies of sets and the distortion constant $C$}
\begin{Def} Let $k,n\in\mathbb N$ and $A$ be a connected component of $F^{-k}(T_n)$. We will call $A$ \emph{a copy of $T_n$ under $F^k$}.\end{Def}
Notice that $X_n$ is a union of all copies of $T_n$ which lie inside $T_1$. We will say that two segments $A$ and $B$ are {\it linked} if $A\cap B\neq \varnothing$ but neither $A$ contains $B$ nor $B$ contains $A$.
\begin{Lm}\label{LmNesting} The set of copies of $T_n,n\in\mathbb Z_+$, is nested, \ie for any copy $A$ of $T_n$ and any copy $B$ of $T_m$, $m,n\in\mathbb Z_+$, one has $A\subset B$, or $B\subset A$, or $A\cap B=\varnothing$.
\end{Lm}
\begin{proof} Given a copy $A$ of $T_n$ under $F^k$ for some $k\in\mathbb N,n\in\mathbb Z_+$ we notice that $F(A)$ does not need to be a copy of $T_n$ under $F^{k-1}$. Indeed, the connected component of $F^{1-k}(T_n)$ might be larger than $F(A)$ if $F$ is not a homeomorphism on $A$, $\ie$ if $0\in A$. Denote this connected component by $\widetilde{F(A)}$. Let us show that for any copy $A$ of $T_n$ under $F^k$ and any copy $B$ of $T_m$ under $F^l$, $m,n\in\mathbb Z_+,l,k\in \mathbb N$, such that $A$ and $B$ are linked, one has $\widetilde{F(A)}$ and $\widetilde{F(B)}$ are also linked. There are a few possibilities.
\vskip 0.2cm
\noindent $a)$ $0\notin A\cup B$. Then $F$ is a homeomorphism on $A\cup B$, therefore $\widetilde{F(A)}= F(A)$ and $\widetilde{F(B)}=F(B)$ are linked.
\vskip 0.2cm
\noindent $b)$ $0\in A\setminus B$ (the case $0\in B\setminus A$ is similar). Without loss of generality we may assume that $B\subset \mathbb R_+$. Clearly, $A\cup B\subset T_1$. Since $F|_{T_1}=f$ is even, $A$ is symmetric with respect to $0$. We have that $A_+:=A\cap \mathbb R_+$ and $B$ are linked, $f$ is a homeomorphism on $A_+\cup B$, therefore $F(A)=f(A_+)$ and $F(B)=f(B)$ are linked. Since $\widetilde{F(A)}\setminus F(A)$ consists of a segment attached to $F(A)$ at $F(0)$, we obtain that $\widetilde{F(A)}$ and $\widetilde{F(B)}$ are linked.
\vskip 0.2cm
\noindent $c)$ $0\in A\cap B$. Then both $A$ and $B$ are symmetric with respect to $0$, therefore cannot be linked.

Notice that $T_1$ cannot be linked with any copy of $T_n$, $n\in\mathbb Z_+$. Therefore, given a copy $A$ of $T_n$ under $F^k$ and a copy $B$ of $T_m$ under $F^l$, $m,n,l,k\in\mathbb Z_+$, such that $A$ and $B$ are linked and $k=0$, we have $n\geqslant 2$. In this case we will view $A=T_n$ as a copy of $T_{n-1}$ with $k=u_n>0$ (the $n$th Fibonacci number). Similarly, if $l=0$. If $k,l>0$ we can replace $A,B$ with new linked copies $\widetilde{F(A)}$ and $\widetilde{F(B)}$. Since one can make this replacement only finitely many times we obtain a contradiction, which completes the proof.
\end{proof}

We will now introduce two types of copies that will play a special role in our construction.

\begin{Def} Let $A$ be a copy of $T_n$ under $F^k$. If $F^j(A)\cap T_n=\varnothing$ for all $0\leqslant j<k$ we will call $A$ \emph{a primitive copy} of $T_n$. Otherwise, let $0\leqslant j<k$ be the maximal number such that $B=F^j(A)\cap T_n\neq \varnothing$.  We will say that $A$ is a \emph{separated copy of $T_n$} if there is no $l$ such that $F_{n-1}^l(B)=F^k(A)$. Here by writing $F_{n-1}^l(B)=F^k(A)$ we mean, in particular, that $F_{n-1}^l(B)$ is well defined, \ie $F_{n-1}^i(B)\subset T_n\cup J_n$ for all $0\leqslant i<l$. \end{Def}

The significance of these two types of copies is made clear in the next Proposition. Before we proceed with the Proposition, however, we require an additional Lemma.
Let $x_n=F^n(0)$, $n\geqslant 0$.
\begin{lem} \label{postcrit}
  Let $\{F,G\}$ be a renormalization period $2$ periodic cycle. Then the following holds.

  \vspace{1.5mm}

  \noindent $1)$ Consider the nested sequence of renormalization intervals $T_n$  for $G$. The nearest postcritical points to the boundaries of $T_{2 n}$, $n \ge 1$, are  $(-\lambda^2)^{n-1} x_4$  and  $(-\lambda^2)^{n-1} x_7=(-\lambda^2)^{n-1}(-\lambda x_4)$,  while those to the boundaries of $T_{2 n+1}$, $n \ge 1$, are  $(-\lambda^2)^{n-1} x_4$  and  $(-\lambda^2)^{n-1} x_{11}=(-\lambda^2)^n x_4$.

  \vspace{1.5mm}

  \noindent $2)$ Consider the nested sequence of renormalization intervals $T_n$  for $F$. The nearest postcritical points to the boundaries of  $T_{2 n+1}$, $n \ge 1$, are  $(-\lambda^2)^{n-1} x_7=  (-\lambda^2)^{n-1} \lambda x_4$  and  $(-\lambda^2)^{n-1} x_{11}=(-\lambda^2)^n x_4$, while those to   $T_{2 n}$, $n \ge 2$, are  $(-\lambda^2)^{n-2} x_7=  (-\lambda^2)^{n-2} \lambda x_4$  and  $(-\lambda^2)^{n-2} x_{18}=(-\lambda^2)^{n-1} \lambda x_4$.
\end{lem}
\begin{proof}
  By a result of Lyubich and Milnor (see \cite{LM}, Section 6), the order of the  postcritical set of maps in $\cA^{-}$ coincides with the order of the postcritical set for unimodal Fibonacci maps.

Then by Lemma $3.5$ in \cite{LM},
    $$\cP \subset [x_1,x_6] \cup [x_{12},x_4] \cup [x_5,x_{13}] \cup [x_{11},x_3] \cup [x_7,x_2].$$
    Additionally, by Lemma $6.1$  in \cite{LM}
    \footnote{We notice that there is slight discrepancy in the notations: our $T_n$ and $J_n$ are denoted by $T_{n-1}$ and $J_{n-1}$ correspondingly in \cite{LM}.}
    \begin{align*}
      [x_5,-x_5] \subset  \ & T_3 \subset [-x_3,x_3] \\
      [x_{11},x_3] \subset \ & J_3.
    \end{align*}
    Since $0<x_{13}<-x_5$, we obtain, therefore, that $x_{13} \in T_3$, and points $x_4$ and $x_{11}=-\lambda^2 x_4$ are the nearest postcritical points to $T_3$.

On the other hand, by Lemma $6.1$,
\begin{equation*}
[-x_3,x_3] \subset  T_2 \subset [-x_2,x_2].
    \end{equation*}
    Therefore, $x_7=-\lambda x_4$ is the nearest postcritical point to the right of $T_2$, and $x_4$ is the nearest postcritical point to the left.

    The claim for $G$ follows. The result for $F$ follows from the facts that:

    \vspace{1mm}

    \noindent $a)$ The interval $T_4$ for $G$ is mapped by $-\lambda^{-1}$ into the interval $T_3$ for $F$, while the nearest postcritical points $x_{18}$ and $x_{11}$ are mapped to $x_{11}$ and $x_{7}$, respectively;

    \vspace{1mm}

    \noindent $b)$ The interval $T_5$ for $G$ is mapped by $-\lambda^{-1}$ into the interval $T_4$ for $F$, while the nearest postcritical points $x_{29}$ and $x_{11}$ are mapped to $x_{18}$ and $x_{7}$, respectively.
\end{proof}

\begin{Prop}\label{PropKoebeSpace} Let $A$ be either a primitive or a separated copy of $T_n$ under $F^k$. Then the inverse branch $\phi$ of $F^k \vert_A$ extends to an analytic map on $(p,q)$, where $p$ and $q$ are the nearest postcritical points to $T_n$, listed in Lemma $\ref{postcrit}$.
\end{Prop}
\begin{proof}  If $0\in F^j(A)$ for some $j<k$ then $A$ is not primitive, so it is separated. Moreover, by nesting property, $F^j(A)\subset T_n$. Since $F_{n-1}$ is the first return map to $T_{n}$, it follows that the iterate $F^{k-j}$ that returns $F^j(A)$ to $T_n$ has to be some  iterate $l$ of the first return map $F_{n-1}$. Therefore, $F^l_{n-1}(F^j(A))=F^k(A)$ and we arrive to a contradiction to separatedness.

Thus, $0\notin F^j(A)$ for all $0\leqslant j<k$.
\end{proof}


We, therefore, obtain with the help of Koebe Distortion Principle:

\begin{Co}\label{CoKoebeDist} There exists a constant $C$ depending only on $d$ such that in the settings of Proposition \ref{PropKoebeSpace} one has:
$$|\phi'(x)|/|\phi'(y)|\leqslant C\;\;\text{for all}\;\;x,y\in T_n.$$
\end{Co}

Clearly, a similar result holds for $G$.

\subsection{Representations of $X_n$ and $X_{n+m}$.}
Fix $n,m\in\mathbb N$. 
Observe that
$$\{x\in T_n:F_{n-1}^k(x)\in T_{n+m}\;\;\text{for some}\;\;k\geqslant 0\}=\lambda^{n-1}X_{m+1}.$$
Clearly, $T_{n+m} \subset \lambda^{n-1}X_{m+1}$.

Denote by $\mathcal P_{n,m},\mathcal S_{n,m}$ and $\mathcal{PS}_{n,m}$ the set of all primitive, separated or both kinds of copies $P$ of $T_n$ under $F^k$, $k\geqslant 0$, such that $P\subset T_1$ and $F^j(P)\not\subset T_{n+m}$ for all $0\leqslant j < k$. In addition, let $\mathcal O_{n,m}$ be the set of all copies $P\subset T_n$ of $T_n$ such that
\begin{itemize} \item{} $F_{n-1}^k(P)=T_n$ for some $k\in\mathbb Z_+$;
\item{} $F_{n-1}^j(P)\not\subset T_{n+m}$ for all $0\leqslant j<k$.
\end{itemize}
For a copy $P$ of $T_n$ under $F^k$ set
\begin{equation}
\label{XP}  X_P=P\cap F^{-k}(\lambda^{n-1}X_{m+1}).
\end{equation}
  Similarly, set $Y_P=P\cap F^{-k}(Y_n)$. Introduce an \emph{extended set of escaping points}
\begin{equation} Z_{n,m}=\bigcup\limits_{P\in \mathcal O_{n,m}}Y_P.
\end{equation}
Thus, $Z_{n,m}$ is the set of all points in $T_n$ that exit $T_n$  after some number of high-order iterates, to never return, without ever passing through the central subinterval $T_{n+m}$.



Finally, for a copy $Q$ of $T_n$ under $F^k$ set
\begin{equation}
 \label{ZQ} Z_Q=Q\cap F^{-k}(Z_{n,m})
\end{equation}

The next three lemmas are a preparation for the criterion of the existence of a wiled attractor, proved in Subsection $\ref{criterion}$.

\begin{Lm}\label{LmXnm} One has:
$$X_{n+m}=\bigcup\limits_{P\in\mathcal{PS}_{n,m}}X_P.$$
\end{Lm}
\begin{proof} Set $X=\bigcup\limits_{P\in\mathcal{PS}_{n,m}}X_P$. Notice that $X_{n+m}$ is the union of all copies of $T_{n+m}$ lying inside $T_1$. At the same time,  $\lambda^{n-1}X_{m+1}$ is the union of all copies of $T_{n+m}$ in $T_n$, and hence, is a subset of $X_{n+m}$, and so is $X$, being a subset of points in copies of $T_{n+m}$ in $T_1$: $X \subset  X_{n+m}$.

We will now show that $X_{n+m} \subset X$.   Let $Q$ be a copy of $T_{n+m}$ under $F^k$.

 Assume first that for the smallest $l\geqslant 0$ with $F^l(Q)\subset T_n$ one has $F^l(Q)\subset \lambda^{n-1}X_{m+1}$. Let $P$ be the copy of $T_n$ under $F^l$ containing $Q$. Then by nesting property and the choice of $l$ we have $F^j(P)\cap T_n=\varnothing$ for $0\leqslant j<l$, therefore $P$ is primitive and, by assumption, $Q\subset X_P\subset X$. 

  Assume now that $F^l(Q)\not\subset \lambda^{n-1}X_{m+1}$. Since $Q$ is a copy of $T_{n+m}$ this implies that $F^l(Q)\cap \lambda^{n-1}X_{m+1}=\varnothing$. Since $F^k(Q) \subset T_{n+m}\subset \lambda^{n-1}X_{m+1}$,
there exists $j \le k$  such that
  \begin{equation}
\label{cont3}    F^j(Q)\subset \lambda^{n-1}X_{m+1}.
  \end{equation}
  Assume that $j$ is the minimal number with this property. Denote by $P$ the copy of $T_n$ under $F^j$ containing $Q$. Since $j$ is the minimal integer for which $(\ref{cont3})$ holds, $F^i(Q) \not \subset T_{n+m}$ for $i < j$. Therefore,  $P \in \cS_{n,m}$ and   $Q\subset X$.
\end{proof}

\begin{Lm}\label{LmDisj0} The elements of the sets from the collection 
$\{X_P\}_{P\in\mathcal{PS}_{n,m}}\cup \{Z_P\}_{P\in\mathcal{PS}_{n,m}}$ are pairwise disjoint subsets of $X_n$.
\end{Lm}
\begin{proof} For convenience, the proof is split into three parts.

  \vspace{1mm}

\noindent   {\bf Part I.} Let us consider now the simplest case $P=T_n$ (the primitive copy of $T_n$ under $F^0=\id$). By definition, we have $X_{T_n}=\lambda^{n-1}X_{m+1}$, $Z_{T_n}=Z_{n,m}$. Thus, we need to show that $\lambda^{n-1}X_{m+1}$ is disjoint from $Z_{n,m}$. Assume the contrary. Then there exists $z\in T_n$ and  $r,t\geqslant 0$ such that
  \begin{equation}
  \label{cont4}  F_{n-1}^r(z)\in T_{n+m}\text{ and }F_{n-1}^t(z)\in Y_n.
  \end{equation}
Assume that $r$ and $t$ are minimal such that  $(\ref{cont4})$ holds.  Moreover, let $P$ be the copy of $T_n$ under $F_{n-1}^t$ containing $z$. Then $F_{n-1}^j(P)\not\subset T_{n+m}$ for all $0\leqslant j\leqslant t$.  The definition of $Y_n$ implies that $t\geqslant r$. Since $F_{n-1}^r(P)\not\subset T_{n+m}$ and $F_{n-1}^r(P)\cap T_{n+m}$ contains $z$ (and so is nonempty) by nesting property we have $F_{n-1}^r(P)\supsetneq T_{n+m}$. Observe that the composition $F_{n}\circ F_{n+1}\circ \cdots \circ F_{n+m-1}$ maps $T_{n+m}$ into $T_n$ and $\partial T_{n+m}$ into $\partial T_n$. In addition, this composition can be written as $F_{n-1}^l$ for some $l$. On one hand, we obtain that $F_{n-1}^{r+l}(z)\in T_n$ and so $r+l\leqslant t$. On the other hand, $\partial T_{n+m}$ are interior points of $F_{n-1}^r(P)$, therefore one of the points from $\partial T_n$ is an interior point of $F_{n-1}^{r+l}(P)$. This contradicts to the fact that $F_{n-1}^{t-r-l}$ is defined on $F_{n-1}^{r+l}(P)$. This contradiction shows that $\lambda^{n-1}X_{m+1}$ is disjoint from $Z_{n,m}$.

  \vspace{1mm}

\noindent {\bf Part II.} Next, let us show that for any separated copy $Q\subset T_n$ of $T_n$ under $F^k$ such that $Q\in\mathcal S_{n,m}$ one has $Q\cap (\lambda^{n-1}X_{m+1}\cup Z_{n,m})=\varnothing$. Assume the contrary. Let $z\in Q\cap (\lambda^{n-1}X_{m+1}\cup Z_{n,m})$.
\vskip 0.2cm\noindent
       {\bf a)} $z\in \lambda^{n-1}X_{m+1}$. Then $F_{n-1}^r(z)\in T_{n+m}$. By definition, one can write $F_{n-1}^r$ as a restriction of $F^l$ for some $l$.

       \vspace{1mm}

       \noindent $a1)$ Assume that $k\geqslant l$. If $F_{n-1}^r(Q)$ intersects the postcritical set of $F_{n-1}$ taking into account that $F_{n-1}$ is the first-return map from $T_n \cup J_n$ we obtain that $T_n \supset F^k(Q)=F_{n-1}^s(Q)$ for some $s\geqslant r$ which implies that $Q$ is not separated. In the opposite case by nesting property we have $F_{n-1}^r(Q)\subset T_{n+m}$ which implies that $Q\notin \mathcal{S}_{n,m}$.

       \vspace{1mm}

       \noindent $a2)$ Assume that $k<l$. Then again $T_n=F^k(Q)=F_{n-1}^s(Q)$ for some $s$ and $Q$ is not separated.
\vskip 0.2cm\noindent
{\bf b)} $z\in Z_{n,m}$. Then $F_{n-1}^r(z)\in Y_n$ for some $r\geqslant 0$. Let $l>0$ be such that $F_{n-1}^r$ is a restriction of $F^l$. By definition of $Y_n$, since $F^k(z)\in F^k(Q)\subset T_n$ we obtain that $k\leqslant l$. Since $F_{n-1}$ is the first-return map on $T_n\cup J_n$ it follows that $F^k(Q)=F_{n-1}^s(Q)$ for some $s\leqslant r$. This implies that $Q$ is not separated.
\vskip 0.2cm\noindent
Thus, in all cases we obtain a contradiction which shows that $Q\cap (\lambda^{n-1}X_{m+1}\cup Z_{n,m})=\varnothing$ for all $Q\in\mathcal S_{n,m}$.



  \vspace{1mm}

\noindent {\bf Part III.} Now we are ready to prove the general case of Lemma \ref{LmDisj0}. Let $P,Q\in\mathcal{PS}_{n,m}$ be copies of $T_n$ under $F^k$ and $F^l$ respectively and $A,B\in\{X,Z\}$ be such that $A_P\cap B_Q\neq\varnothing$ (here, $A_P \in \{X_P,Z_P\}$ and $B_Q \in \{X_Q,Z_Q\}$, see $(\ref{XP})$ and $(\ref{ZQ})$).

Without loss of generality we may assume that $k\geqslant l$. Observe that $F^l(B_Q)\subset B_{T_n}$, where $B_{T_n}=\lambda^{n-1}X_{m+1}$ by $(\ref{XP})$, if $B=X$, and $B_{T_n}=Z_{n,m}$ by $(\ref{ZQ})$, if $B=Y$. We have that $F^l(P)$ is a subset of a copy $R$ of $T_n$ under $F^{k-l}$. The definitions of primitive and separated copies imply that either $R$ is separated and moreover $R\in\mathcal S_{n,m}$, or $R=T_n$ and $k=l$. In the first case $R\cap B_{T_n}=\varnothing$ by Part II. In the second case $A_R=A_{T_n}$ coincides with either  $\lambda^{n-1}X_{m+1}$ or $Z_{n,m}$. By Part I from $A_{T_n}\cap B_{T_n}\neq \varnothing$ we conclude that $A=B$. Since $k=l$ and $P\cap Q\neq\varnothing $ it follows that $P=Q$ and thus $A_P=B_Q$. This completes the proof.

\end{proof}
\begin{Lm}\label{LmXnIsUnion} The subset
\begin{equation*} \bigcup\limits_{P\in\mathcal{PS}_{n,m}}X_P\cup \bigcup\limits_{P\in\mathcal{PS}_{n,m}}Z_P\subset X_n
\end{equation*} has full Lebesgue measure in $X_n$
\end{Lm}
\begin{proof} Notice that for almost all points $z\in T_1\cup J_1$ the iterate of $z$ either approaches the critical point arbitrarily close or escapes. Let $z\in X_n$ be such that one of the above two options happen.
 \vskip 0.2cm\noindent
        {\bf a)} $F^k(z)\in T_{n+m}$ for some $k\in\mathbb N$. Let $k$ be the minimal such number. Observe that the set of $l\geqslant 0$ such that $F^l(z)\in T_n$ and  there exists $r\geqslant 0$ with $F_{n-1}^r(F^l(z))=F^k(z)$ is non-empty (indeed, the equality is true for $l=k$ and $r=0$). Fix $l$ to be the minimal number in this set. Evidently, $z$ belongs to either primitive or separated copy $P$ of $T_n$ under $F^l$.
        Moreover,
\begin{equation*} F^l(z)\in \left(\bigcup\limits_{t\geqslant 0}F_{n-1}^{-t}(T_{n+m})\right)\cap T_n=\lambda^{n-1}X_{m+1}.
\end{equation*} Thus, $z\in X_P$. Since $k$ is minimal with $F^k(z)\in T_{n+m}$ we have that $F^j(P)\not\subset T_{n+m}$ for $0\leqslant l$. Therefore, $P\in\mathcal{PS}_{n,m}$.
\vskip 0.2cm\noindent
{\bf b)} $F^k(z)\not\in T_{n+m}$ for all $k\geqslant 0$ and $F^l(z)\not\in T_0$ for some $l\in\mathbb N$, \ie $z$ is escaping. Since $z\in X_n$ there exists $k\geqslant 0$ such that $F^k(z)\in T_n$. Let $k$ be the maximal such number. Then $F^k(z)\in Y_n$. As in the case $a)$ find minimal $l\geqslant 0$ such that $F_{n-1}^r(F^l(z))=F^k(z)$ for some $r\geqslant 0$ and $F^l(z)\in T_n$. Then $z$ belongs to a copy $P$ of $T_n$ under $F_l$ such that $P$ is either primitive or separated. In addition, $F^j(z)\notin T_{n+m}$ for all $j\geqslant 0$ implies that $P\in\mathcal{PS}_{n,m}$. Let $Q$ be the copy of $T_n$ under $F_{n-1}^r$ containing $w=F^l(z)$. We have $Q\subset T_n$ and $F_{n-1}^r(Q)=T_n$. The conditions of the case $b)$ imply that $F_{n-1}^j(w)\notin T_{n+m}$. Thus, $w\in Z_{n,m}$ and $z\in Z_P$.
\end{proof}

\subsection{Criterion for existence of a wild attractor } \label{criterion}
In this section we will derive some recursive estimates for $\eta_n$ and $\zeta_n$, and use them to formulate a criterion for the existence of a wild attractor.

Set $\zeta_{n,m}=|Z_{n,m}|/|T_n|$.
\begin{Th}\label{RecursiveThm} Let $C$ be the constant from Corollary \ref{CoKoebeDist}. Then for any $n,m\in\mathbb N$ one has:
\begin{equation}\label{EqRecursiveMain}\frac{\eta_n\eta_{m+1}}{\eta_{m+1}+C\zeta_{n,m}}\leqslant \eta_{n+m}\leqslant \frac{\eta_n\eta_{m+1}}{\eta_{m+1}+C^{-1}\zeta_{n,m}}.\end{equation}
\end{Th}
\begin{proof} By Lemma \ref{LmXnm} we have:
$$X_{n+m}=\bigcup\limits_{P\in\mathcal{PS}_{n,m}}X_P.$$ Using Lemma \ref{LmDisj0} we obtain that $$|X_n|= \sum\limits_{P\in\mathcal{PS}_{n,m}} (|X_P|+|Z_P|),$$ where, as before, $|\cdot|$ stands for the Lebesgue measure.
From Corollary \ref{CoKoebeDist} we deduce that for every $P\in\mathcal{PS}_{n,m}$ one has:
$$ C|Z_{n,m}|/|\lambda^{n-1}X_{m+1}|\geqslant |Z_P|/|X_P|\geqslant C^{-1}|Z_{n,m}|/|\lambda^{n-1}X_{m+1}|.$$
It follows that \begin{equation*}|X_n|\leqslant \sum\limits_{P\in\mathcal{PS}_{n,m}}|X_P|\cdot\left(1+C\frac{|Z_{n,m}|}{\lambda^{n-1}|X_{m+1}|}\right)=|X_{n+m}|
\cdot\left(1+C\frac{\zeta_{n,m}}{\eta_{m+1}}\right).
\end{equation*} We arrive at:
\begin{equation*}
\frac{\eta_{n+m}}{\eta_n}=  \frac{|X_{n+m}|}{|X_n|}\geqslant \frac{\eta_{m+1}}{\eta_{m+1}+C\zeta_{n,m}},
\end{equation*} which proves the left hand side inequality of \eqref{EqRecursiveMain}. The right hand side is proven similarly.
\end{proof}
Recall that \begin{equation}Z_n=\left(\bigcup\limits_{k\geqslant 0}F_{n-1}^{-k}(Y_n)\right)\cap T_n,\;\;\zeta_n=\frac{|Z_n|}{|T_n|}.\end{equation}
Observe that $Z_n\supset Z_{n,m}$ for all $m$ and so $\zeta_n\geqslant \zeta_{n,m}$.

Theorem $\ref{RecursiveThm}$ allows one to provide  criteria for existence and non-existence of the wild attractor.

Out first criterion is the following

\begin{Prop}\label{PropWild} Assume that $\eta_n/\zeta_n>C$ for some $n$. Then $\inf\limits_{n\in\mathbb N}\eta_n>0$ and $\lim\limits_{n\to\infty}\zeta_n=0$. Moreover, the Fibonacci map under consideration has a wild attractor.
\end{Prop}
\begin{proof} Let $\gamma=C\zeta_n/\eta_n<1$ for some $n$. Fix this number $n$. By \eqref{EqRecursiveMain}, for any $m$ we have:
\begin{equation*} \eta_{n+m}\geqslant \frac{\eta_{m+1}}{\frac{\eta_{m+1}}{\eta_n}+C\frac{\zeta_{n,m}}{\eta_n}}\geqslant \frac{\eta_{m+1}}{\frac{\eta_{m+1}}{\eta_n}+\gamma}.\end{equation*} If $\inf\eta_k=0$ then for some $m$ we have $\eta_{m+1}/\eta_n+\gamma<1$ and therefore $\eta_{n+m}>\eta_{m+1}$ which is impossible, since $\eta_k$ is decreasing. Thus, $\inf\eta_k>0$. This means that $\cap_{k\in\mathbb N} X_k$ has positive measure. But this set is precisely the set of points in $T_1$ whose orbits approach the critical point arbitrarily close.

Further, since $\eta_k$ is non-increasing, there exists $\eta=\lim\limits_{k\to\infty}\eta_k>0$. From \eqref{EqRecursiveMain} we have:
$$C^{-1}\zeta_{k,m}\leqslant \frac{\eta_k\eta_m}{\eta_{k+m}}-\eta_m.$$ Taking limit when $m\to\infty$ we obtain $C^{-1}\zeta_k\leqslant \eta_k-\eta$. It follows that $\lim\limits_{k\to \infty}\zeta_k=0$. This completes the proof.
\end{proof}

On the other hand, we have the following

\begin{Prop}\label{NoWildProp} Assume that $\eta_n/\zeta_n<C^{-1}$ for some $n$. Then $\eta_n$ converges to zero exponentially fast when $n\to\infty$ and $\inf\limits_{n\in\mathbb N}\zeta_n>0$. Moreover, Fibonacci map under consideration does not have a wild attractor.
\end{Prop}
\begin{proof} Fix $n$ and $\gamma$ such that $C\eta_n/\zeta_n<\gamma<1$. Observe that $\zeta_{n,m}$ is increasing in $m$ and converges to $\zeta_n$ when $m\to\infty$. In particular, there exists $m_0$ such that $C\eta_n/\zeta_{n,m}<\gamma$ for all $m\geqslant m_0$. By \eqref{EqRecursiveMain} we have for all $m\geqslant m_0$:
\begin{equation*} \eta_{n+m}\leqslant \frac{\eta_{m+1}}{\frac{\eta_{m+1}}{\eta_n}+C^{-1}\frac{\zeta_{n,m}}{\eta_n}}\leqslant \gamma\eta_{m+1}.\end{equation*} This implies that $\eta_k$ converges to zero (exponentially fast) when $k\to\infty$. It follows that $\cap_{k\in\mathbb N} X_k$ has zero Lebesgue measure and the Fibonacci map does not have a wild attractor.

Further, by left hand side of \eqref{EqRecursiveMain} we have for any $k,m\in \mathbb N$:
 \begin{equation}\label{EqZetakFromBelow}C\zeta_k\geqslant C\zeta_{k,m}\geqslant \left(\frac{\eta_k}{\eta_{k+m}}-1\right)\eta_{m+1}.\end{equation} Let $R=\liminf\limits_{k\to\infty}\eta_k^{1/k}$. Observe that $R>0$, since $\eta_{k+m}\geqslant C^{-1}\eta_k\eta_m$ for every $k,m\in\mathbb N$ (again, by left hand side of \eqref{EqRecursiveMain}). Fix $m$ such that $R^{-m}>2$. Then for infinitely many $k\in\mathbb N$ we have $\eta_k/\eta_{k+m}\geqslant 2$. For such $k$ from \eqref{EqZetakFromBelow} we obtain $C\zeta_k\geqslant \eta_{m+1}>0$. This shows that $\inf\zeta_k>0$ and finishes the proof.
\end{proof}

In the remaining case for the behavior of $\eta_n/\zeta_n$ we obtain the following.

\begin{Prop}\label{NoWildProp1} Assume that $\eta_n/\zeta_n$ is bounded from above. Then $\eta_n\to 0$ and there is no wild attractor.
\end{Prop}
\begin{proof} Assume that $\eta_n$ does not converge to zero when $n\to\infty$. Since $\eta_n$ is non-increasing there exists $\eta=\lim\limits_{n\to\infty}\eta_n>0$. Given that $\eta_n/\zeta_n$ is bounded from above and $\zeta_n$ is non-increasing we obtain that there exists a limit $\zeta=\lim\limits_{n\to\infty}\zeta_n>0$. Taking limit in \eqref{EqRecursiveMain} first when $m\to\infty$ and then when $n\to\infty$ we obtain:
$$\eta\leqslant\frac{\eta^2}{\eta+C^{-1}\zeta}.$$ Thus, $\eta+C^{-1}\zeta\leqslant \eta$. This contradiction finishes the proof.
\end{proof}
\begin{Prop}\label{NoWildProp1} Assume that $\eta_n/\zeta_n$ is bounded from below. Then $\zeta_n\to 0$.
\end{Prop}
\begin{proof}
Since $\zeta_n$ is non-increasing, if it does not converge to zero then $\inf\limits_{n\in\mathbb N}\zeta_n>0$ and therefore  $\inf\limits_{n\in\mathbb N}\eta_n>0$. However, from the arguments of the second part of the proof of Proposition \ref{PropWild} we have that $\inf\limits_{n\in\mathbb N}\eta_n>0$ implies $\zeta_n\to 0$ when $n\to\infty$. This contradiction completes the proof.
\end{proof}

\subsection{Estimating $\eta_n$ and $\zeta_n$}
 \begin{Prop} Modulo countable set of points, the set $T_n\setminus Z_n$ consists of the points $z\in T_n$ such that either 
 \\ $a)$ $F_{n-1}^j(z)\in T_n\cup J_n$ for all $j$ or 
 \\ $b)$ there exists $l$ such that $F^l(F_{n-1}^j(z))\in T_n$, where $j$ is the minimal number such that $F_{n-1}^j(z)\notin T_n\cup J_n$.
 \end{Prop}
 \begin{proof} Let $S_n$ be the set of points $z\in T_n$ satisfying either $a)$ or $b)$. Let $z\in Z_n$. Then by definition $F_{n-1}^k(z)\in Y_n$ for some $k\geqslant 0$.

   Assume that $z$ satisfies $a)$. From the definition of $Y_n$ we obtain that $F^l(F^k_{n-1}(z))\notin T_n$ for all $l>0$. It follows that $F_{n-1}^{k+i} \in J_n$ for all $i>0$. Therefore, $F_{n-1}^{k+1}(z)$ is the unique repelling fixed point of $F_{n-1}$ in $J_n$.

 Assume now that $z$ satisfies $b)$. Clearly, $j\geqslant k$. Then $F^l(F_{n-1}^j(z))\in T_n$ contradicts to $F_{n-1}^k(z)\in Y_n$. Thus, $Z_n$ intersects $S_n$ by at most countable set of points.

 Let $z \in T_n\setminus(Z_n\cup S_n)$. Since $z$ does not satisfy $a)$ there exists a maximal number $j$ such that $F_{n-1}^j(z)\in T_n\cup J_n$. Let $k\leqslant j$ be the maximal number such that $F_{n-1}^k(z)\in T_n$. Since $z\notin Z_n$ we have $F_{n-1}^k(z)\notin Y_n$. Therefore, for some $l$ one has $F^l(F_{n-1}^k(z))\in T_n$. Since $F_{n-1}$ is the first return map from $T_n\cup J_n$ to $T_{n-1}$, there exists $p>k$ such that $F^l(F_{n-1}^k(z)=F_{n-1}^p(z) \in T_n$. Since by maximality of $k$ and $j$, for all $k< m \le j$, $F^m_{n-1}(z) \in J_n$, and additionally, $F^{j+1}_{n-1} \notin T_n \cup J_n$, we must have $p>j+1$. Thus, there exists $l_1>0$ such that $F^p_{n-1}(z)=F^{l_1}(F^{j+1}_{n-1}(z)) \in T_n$, and $z$ satisfies $b)$. Therefore  $T_n \setminus(Z_n\cup S_n)$ is empty.  This completes the proof of the proposition.

 \end{proof} 

 Thus, one can estimate $Z_n$ as follows.

\begin{alg}\label{alg1} Given a point $z\in T_n$ apply repeatedly $F_{n-1}$ until large number $N$ times. If $F_{n-1}^j(z)\in T_n\cup J_n$ for all $j\leqslant N $ then discard $z$.

\vspace{1mm}
Otherwise, let $j$ be the minimal number such that $w=F_{n-1}^j(z)\notin T_n \cup J_n$. Apply $F$  repeatedly until large number $L$ times. Suppose $u=F^{l} (w) \notin T_{1} \cup J_{1}$ for some $l \leqslant L$, and suppose that such $l$ is minimal. Then the length of a segment $A(z)$ equal to the connected component of $F^{-j}_{n-1}(F^{-l}(T_0 \setminus (T_1 \cup J_1)))$, containing $z$, is in $Z_n$.

Union of such segments is a lower boud on $Z_n$. Since taking inverses (through a Newton) is computationally prohibitive, in reality, one subdiveds $T_n$ into subintervals $I_n^r$, and checks if the above construction is valid with $z$ relaced by $I^r_n$. Then $I^r_n$ is i $Z_n$. 
\end{alg}

\vspace{2mm}

\begin{alg}\label{alg2}
Alternatively, if $F^{l}(w) \in T_n$ for some $1 \le l \le L$,  then the segment $B(z)$ equal to the connected component  of   $F_{n-1}^{-j}(F^{-l}(T_n)) \cap T_n$ containing $z$,  belongs to $T_n\setminus Z_n$. Union of the segments of the form $B(z)$ gives a lower bound on $T_n\setminus Z_n$.
\end{alg}

\section{Real renormalization} \label{real_ren}

Denote $\cD^\omega(A,B)$ the Banach space of orientation-preserving $C^\omega$-diffeomorphisms of a closed interval $A$ onto a closed interval $B$. If $A=B=[-1,1]$, this space will be simply denoted as $\cD^\omega$. 

We consider a Fibonacci map $F=(g,f)$ on $J \cup T$ given by
\begin{equation}
F(x)= \left\{ \psi(s^{-1}_J(x)), \quad x \in J,  \atop \phi(p_v(s_T^{-1}(x))), x \in T, \right.
\end{equation}
whenever $g$ is orientation-preserving, and
\begin{equation}
F(x)= \left\{ -\psi(s^{-1}_J(x)), \quad x \in J,  \atop \phi(p_v(s_T^{-1}(x))), x \in T, \right.
\end{equation}
whenever $g$ is orientation-reversing. Here $J=[i,j]$ and $T=[-t,t]$ are two closed disjoint subintervals of $[-1,1]$, $s_I/r_I$  denotes an orientation-preserving/-reversing affine map of $I$ onto $[-1,1]$, $(\psi,\phi) \in \cD^\omega \times \cD^\omega$, and $p_v$ is the  ``standard'' power family of degree $d$, normalized so that $p_v(\pm 1)=1$ and parameterized by its critical value:
\begin{equation}
p_v(x)=v+(1-v) |x|^d, \quad v \in [-1,1],
\end{equation}
Restrict $t$ to some interval $A=[\tau_1,\tau_2]$, $i$ to $B=[\iota_1,\iota_2]$ and $j$ to $C=[\zeta_1,\zeta_2]$, then the set of such maps $F$ is isomorphic to the product Banach space
  \begin{equation}
    \cB=[-1,1] \times A \times B \times C \times \cD^\omega \times \cD^\omega.
  \end{equation}

  We will endow the Banach space $\cB$ with the following $l_1$-norm:
  \begin{equation}
    \|F\|_1=|v|+|i|+|j|+\|\psi\|_1+\|\phi\|_1.
  \end{equation}

\subsection{Induced renormalization on $\cB$}
We will now describe how renormalization acts on each factor of the Banach space $\cB$. We denote
\begin{equation*}
  F \equiv (v, i, j, t,\psi, \phi),
\end{equation*}
and
\begin{equation}\label{ren5}
\tilde F =(\tilde v, \tilde i, \tilde j, \tilde t ,\tilde \psi, \tilde \phi)=\cR(v,i,j, t,\psi, \phi)=:\cR(F).
\end{equation}
Additionally, denote  $K =[k_1,k_2]$ - the preimage of $T$ under $\psi \circ s^{-1}_J \co \phi$ (or, under $-\psi \circ s^{-1}_J \co \phi$), and $\tilde J_\mp$ -  the left/right component of $(\phi \co p_v)^{-1}(T)$. Then the induced operator on the diffeomorphic factors $\cR: \cD^\omega \times \cD^\omega \mapsto  \cD^\omega \times \cD^\omega$ is given by
\begin{equation}
\label{tildepsiphi}  (\tilde \psi,\tilde \phi)=\left(r^{-1}_T \co \phi \co p_v \co s_{\tilde J_-}, s^{-1}_T \co \psi \co s^{-1}_J \co \phi \co s_K \right),
\end{equation}
if $g$ is orientation-preserving, and
\begin{align*}
  (\tilde \psi,\tilde \phi)&=\left(r^{-1}_T \co \phi \co p_v \co r_{\tilde J_+}, r^{-1}_T \co - \co \psi \co s^{-1}_J \co \phi \co s_K \right) \\
  &=\left(r^{-1}_T \co \phi \co p_v \co r_{\tilde J_+}, r^{-1}_T \co - \co \psi \co s^{-1}_J \co \phi \co s_K \right)
\end{align*}
if $g$ is orientation-reversing. Since $ p_v \co r_{\tilde J_+} = p_v \co s_{\tilde J_-}$, we get that the induced operator has period $1$;  denoting $\tilde J=[\tilde i, \tilde j]$,  the renormalized sextuple is given by  $(\ref{tildepsiphi})$ together with
\begin{align}
\label{tildev}  \tilde v&=s^{-1}_{K}(v),\\
\label{tildeJ}  \tilde J&=(\phi \co p_v)^{-1}(T)_-, \\
\label{tildet}  \tilde t&=p_v^{-1}(k_2), 
\end{align}
where  $(\phi \co p_v)^{-1}(T)_\pm$ denote right and left components of this preimage.

\subsection{Computation of the new renormalization intervals}
To compute new renormalization intervals, such as  $K$ and $\tilde J$, one needs to solve equations of the form
\begin{equation}
  f(x)-a=0.
\end{equation}
(e.g., $\tilde i$ is found with $f=\phi \co p_v$, $a=t$; similarly for other points). This is done as a fixed point problem for the Newton map
\begin{equation}
\label{N} N(x)=x-{f(x)-a \over f'(x_0)},
\end{equation}
where $x_0$ is some approximation of $x$.  Specifically, we use the following version of the Banach fixed point theorem to get bounds on $x$:

\begin{thm}({\it Contraction Mapping Principle})
Let $\cD$ be a Banach space. Suppose that the operator $N$ is well-defined and analytic as a map from  $\cN \subset \cD$ to $\cD$. Let $x_0\in \cN$ and $B_\delta(x_0) \subset \cN$ (an open ball of radius $\delta$ around $x_0$) be such that
$$\| D N[x]\| \le  D <1, $$
for any $x \in  B_\delta(x_0)$, and
$$\|N[x_0]-x_0\|\le \epsilon.$$
If $\epsilon < (1-D) \delta$ then the operator $N$ has a fixed point $x^*$ in $B_\delta(x_0)$, such that
$$\| x^*-x_0\| \le {\epsilon \over 1-D }.$$
\end{thm}
For the choice of the operator $N$ as in $(\ref{N})$,
\begin{equation}
DN(x)=1-{f'(x) \over f'(x_0)}.
\end{equation}

\subsection{Bounds on the renormalization derivative $D \cR(F) h_1$} \label{der1}
We choose the following monomial basis in $\cB$:
\begin{align*}
  h_1 &= (1,0,0,0; 0, \ldots; 0, \ldots),\\
  h_2 &= (0,1,0,0; 0, \ldots; 0, \ldots),\\
  h_3 &= (0,0,1,0; 0, \ldots; 0, \ldots),\\
  h_4 &= (0,0,0,1; 0, \ldots; 0, \ldots),\\
  \eta_{k}&=(0,0,0,0; 0, \ldots, e_k, \ldots; 0,  \ldots), \\
  \varphi_{k}&=(0,0,0,0; 0, \ldots;  0, \ldots, e_k, \ldots), 
\end{align*}
where $e_k(x)=x^k$, and, first,  consider the action of the renormalization derivative on the vector  $h_1$. 

\vspace{1mm}

\noindent $1)$ We start by  computing $D \tilde v h_1$. Since $K$ is independent of $v$, we get
\begin{equation}
\tag{$D \tilde v h_1$}\label{Dvh1} \boxed{ D \tilde v h_1={2 \over k_2 - k_1}.}
\end{equation}

\vspace{1mm}

\noindent $2)$ We can use the expression $(\ref{tildeJ})$ to compute the respective derivatives of  $\tilde i$ and $\tilde j$:
\begin{align}
\tag{$D \tilde i h_1$}  \label{Dih1} &\boxed{ D \tilde i h_1 = {1-\phi^{-1}(t) \over d ({\phi^{-1}(t)-v})^{d-1 \over d} (1-v)^{d+1 \over d}},} \\
\tag{$D \tilde j h_1$} \label{Djh1} &\boxed{D \tilde j h_1  = {1-\phi^{-1}(-t) \over d ({\phi^{-1}(-t)-v})^{d-1 \over d} (1-v)^{d+1 \over d}}.}
\end{align}

\vspace{1mm}

\noindent $3)$ Since $\tilde t=p_v^{-1}(k_2)$,
\begin{equation}
\tag{$D \tilde t h_1$} \label{Dth1} \boxed{D \tilde t h_1={k_2-1  \over  d (k_2-v)^{d-1 \over d}  (1-v)^{d+1 \over d}}.}
\end{equation}

\vspace{1mm}

\noindent $4)$ We have, using $(\ref{tildepsiphi})$,
\begin{equation}
  \tag{$D \tilde \psi h_1$} \label{Dtpsih1}  \boxed{
\begin{array}{l}
  (D \tilde \psi h_1)(x) =-{\phi'(p_v(s_{\tilde J}(x))) \over \lambda_T} \left( 1-s_{\tilde J}(x)^d+\right.\\
\hspace{35mm} \left.  +d (1-v) s_{\tilde J}(x)^{d-1} (D s_{\tilde J} h_1)(x) \right),
\end{array}
  } 
\end{equation}
where
\begin{equation}
  \label{DsJh1} (D s_{\tilde J} h_1)(x) ={D \tilde j h_1 - D \tilde i h_1 \over 2} x +{D \tilde j h_1+ D \tilde i h_1 \over 2}.
\end{equation}
Furthermore, since $\tilde \phi$ is independent of $v$,
\begin{equation}
\tag{$D \tilde \phi h_1$} \label{Dphih1} \boxed{D \tilde \phi h_1=0.}
\end{equation}

\vspace{1mm}

\subsection{Bounds on the renormalization derivatives $D \cR(F) h_2$ and $D \cR(F) h_3$} \label{der2}  First, notice, that $\tilde v$, $\tilde J$ and $\tilde \psi$ are independent of $i$ and $j$. The remaining derivatives will be computed below.

\vspace{1mm}

\noindent $1)$ We have $\tilde t=p_v^{-1}(k_2)$, therefore
\begin{equation}
\tag{$D \tilde t h_2$} \label{Dth2} \boxed{D \tilde t h_2= { D k_2 h_2 \over d (1-v)^{1 \over d}  (k_2 -v)^{d-1 \over d} },} 
\end{equation}
where
\begin{align}
\label{Dk2h2}  D k_2 h_2&=(\phi^{-1})'(s_J(\psi^{-1}(t))) (D s_J h_2) \co \psi^{-1}(t), \\
\label{DsJh2}  (D s_J h_2)(x)&= {1-x \over 2}.
\end{align}
Similarly,
\begin{equation}
\tag{$D \tilde t h_3$} \label{Dth3} \boxed{D \tilde t h_3= { D k_2 h_3 \over d (1-v)^{1 \over d}  (k_2 -v)^{d-1 \over d} },} 
\end{equation}
where
\begin{align}
\label{Dk2h3}  D k_2 h_3&=(\phi^{-1})'(s_J(\psi^{-1}(t))) (D s_J h_3) \co \psi^{-1}(t), \\
\label{DsJh3}  (D s_J h_3)(x)&= {1+x \over 2}.
\end{align}

\vspace{1mm}

\noindent $2)$ Notice that $\tilde \psi$ is independent of $i$ and $j$.

\begin{equation}
\tag{$D \tilde \psi h_2$} \label{Dtpsi2} \boxed{D \tilde \psi h_2= 0}.
\end{equation}

\begin{equation}
\tag{$D \tilde \psi h_3$} \label{Dtpsi3} \boxed{D \tilde \psi h_3= 0}.
\end{equation}

\vspace{1mm}

\noindent $3)$ The last derivatives:

\begin{equation}
  \tag{$D \tilde \phi h_2$} \label{Dtphi2} \boxed{
\begin{array}{ l}
  D \tilde \phi h_2= {\psi'(s_J^{-1}(\phi(s_K(x)))) \over \lambda_T} \left(  D s_J^{-1}h_2(\phi(s_K(x))) +\right.\\
  \vspace{1mm}
  \hspace{25mm}+\left.{ \phi'(s_K(x))  \over \lambda_J }   \left({Dk_2h_2-Dk_1h_2 \over 2    }x+ {Dk_2h_2+Dk_1h_2 \over 2    }  \right) \right)
\end{array}
  }
\end{equation}
where
\begin{equation}
 (D s_J^{-1}h_2)(x)={2 x   \over (j-i)^2} - {2 j \over (j-i)^2},
\end{equation}
and $D k_2 h_2$, $D k_1 h_2$ are as in $(\ref{Dk2h2})$.

Similarly,
\begin{equation}
  \tag{$D \tilde \phi h_3$} \label{Dtphi3} \boxed{
\begin{array}{ l}
  D \tilde \phi h_3= {\psi'(s_J^{-1}(\phi(s_K(x)))) \over \lambda_T} \left(  D s_J^{-1}h_3(\phi(s_K(x))) +\right.\\
  \vspace{1mm}
  \hspace{25mm}+\left.{ \phi'(s_K(x))  \over \lambda_J }   \left({Dk_2h_3-Dk_1h_3 \over 2    }x+ {Dk_2h_3+Dk_1h_3 \over 2    }  \right) \right)
\end{array}
  }
\end{equation}
where
\begin{equation}
 (D s_J^{-1}h_3)(x)=-{2 x   \over (j-i)^2} + {2 i \over (j-i)^2},
\end{equation}
and $D k_2 h_3$, $D k_1 h_3$ are as in $(\ref{Dk2h3})$.

\vspace{1mm}

\subsection{Bounds on the renormalization derivative $D \cR(F) h_4$} \label{der3} $\phantom{aa}$\\

\vspace{1mm}

\noindent $1)$ Since $s^{-1}_K(v)=2/( k_2 - k_1) v -(k_2 + k_1)/(k_2-k_1)$,
\begin{equation}
\tag{$D \tilde v h_4$} \label{Dtvh4} \boxed{D \tilde v h_4=2{  k_1 D k_2 h_4-k_2 D k_1 h_4 -v (D k_2 h_4 - D k_1 h_4)  \over (k_2-k_1)^2 },}
\end{equation}
where
\begin{align}
  \label{Dk1h4}  D k_1 h_4 & = -\lambda_J (\phi^{-1})'(s_J(\psi^{-1}(-t)))  (\psi^{-1})'(-t), \\
  \label{Dk2h4}  D k_2 h_4 & = \lambda_J (\phi^{-1})'(s_J(\psi^{-1}(t)))   (\psi^{-1})'(t).
\end{align}

\vspace{1mm}

\noindent $2)$  The computation of  $D \tilde i h_4$ and  $D \tilde j h_4$ is straightforward:
\begin{align}
\tag{$D \tilde i h_4$ }  \label{Dih4}  & \boxed{D \tilde i h_4  =  -{(\phi^{-1})'(t)  \over d (1-v)^{1 \over d} (\phi^{-1}(t) -v)^{d-1 \over d} },} \\
\tag{$D \tilde j h_4$} \label{Djh4}  & \boxed{D \tilde j h_4  = {(\phi^{-1})'(-t) \over d (1-v)^{1 \over d} (\phi^{-1}(-t) -v)^{d-1 \over d} }.}
\end{align}

\vspace{1mm}

\noindent $3)$ Next,
\begin{equation}
\tag{$D \tilde t h_4$} \label{Dth4} \boxed{D \tilde t h_4={ D  k_2  h_4  \over d (1-v)^{1 \over d} (k_2-v)^{d - 1 \over d} }.}
\end{equation}

\vspace{1mm}

\noindent $4)$ Lastly,
\begin{equation}
\tag{$D \tilde \psi h_4$} \label{Dth4} \boxed{D \tilde \psi h_4=(D s_T^{-1} h_4) \hspace{-0.5mm} \co \hspace{-0.5mm} \phi \hspace{-0.5mm} \co \hspace{-0.5mm}  p_v \hspace{-0.5mm} \co \hspace{-0.5mm} s_{\tilde J} + {(\phi' \hspace{-0.5mm}  \co \hspace{-0.5mm}  p_v \hspace{-0.5mm} \co  \hspace{-0.5mm} s_{\tilde J}) \over t } p'_v \hspace{-0.5mm} \co \hspace{-0.5mm} s_{\tilde J} (D s_{\tilde J} h_4),}
\end{equation}
where
\begin{align*} 
(D s_T^{-1} h_4)(x)&=-{x \over t^2}, \\
(D s_{\tilde J} h_4)(x) &= {D \tilde j h_4 - D \tilde i h_4 \over 2} x +{D \tilde j h_4 +D \tilde i h_4 \over 2}.
\end{align*}

\vspace{1mm}

\subsection{Bounds on the renormalization derivative $D \tilde F \eta_k$} \label{der4} First,
\begin{align*}
  (\psi + \eps e_k )^{-1} & =\psi^{-1} \co (id + \eps e_k \co \psi^{-1})^{-1} +o(\eps) \\
  & =\psi^{-1} \co (id - \eps e_k \co \psi^{-1}) +o(\eps) \\
  &=\psi^{-1}- \eps (\psi ^{-1})' e_k \co \psi^{-1} + o(\eps).
\end{align*}
Notice that $\tilde J$ is independent of $\psi$, therefore,
\begin{align*}
\tag{$D \tilde i \eta_k$} &\boxed{ D \tilde i \eta_k  =0,} \\
\tag{$D \tilde j \eta_k$} &\boxed{  D \tilde j \eta_k  =0.}
\end{align*}
We proceed with the computation of the remaining derivatives

\vspace{1mm}

\noindent $1)$ According to the definition $(\ref{tildev})$:
\begin{equation}
\tag{$D \tilde v \eta_k $} \label{Dvh5} \boxed{ D \tilde v \eta_k=2{k_1 D k_2 \eta_k-k_2 D k_1 \eta_k -v (D k_2 \eta_k - D k_1 \eta_k)  \over (k_2-k_1)^2 },}
\end{equation}
and
\begin{align}
  \label{Dk1h5}  D k_1 \eta_k & = -\lambda_J (\phi^{-1})'(s_J(\psi^{-1}(-t)))  (\psi ^{-1})'(-t) e_k(\psi^{-1}(-t)), \\
  \label{Dk2h5}  D k_2 \eta_k & = -\lambda_J (\phi^{-1})'(s_J(\psi^{-1}(t)))  (\psi ^{-1})'(t) e_k(\psi^{-1}(t)).
\end{align}

\vspace{1mm}

\noindent $2)$ Next,
\begin{equation}
\tag{$D \tilde t \eta_k $} \label{Dth5} \boxed{ D \tilde t \eta_k= {D k_2 \eta_k \over d (1-v)^{1 \over d}  (k_2 -v)^{d-1 \over d} }.} 
\end{equation}

\vspace{1mm}

\noindent $3)$ Also,

\begin{equation}
\tag{$D \tilde \psi \eta_k$} \label{DtpsietaK} \boxed{ D \tilde \psi \eta_k=0,}
\end{equation}

\noindent $4)$ Finally,

\begin{equation}
\tag{$D \tilde \phi \eta_k$} \label{DtpsietaK} \boxed{ D \tilde \phi \eta_k={e_k \hspace{-0.5mm} \co \hspace{-0.5mm} s_J^{-1} \hspace{-0.5mm} \co \hspace{-0.5mm} \phi \hspace{-0.5mm} \co \hspace{-0.5mm} s_K  \over \lambda_T}+ {D s_K \eta_k  \over \lambda_T \lambda_J } (\psi' \hspace{-0.5mm} \co \hspace{-0.5mm} s_J^{-1} \hspace{-0.5mm} \co \hspace{-0.5mm} \phi \hspace{-0.5mm} \co \hspace{-0.5mm} s_K) (\phi' \hspace{-0.5mm} \co \hspace{-0.5mm} s_K),}
\end{equation}
where
\begin{equation}
(D s_K \eta_k)(x) ={D k_2 \eta_k -D k_1 \eta_k  \over 2     } x + {D k_2 \eta_k +D k_1 \eta_k  \over 2     }.
  \end{equation}

\vspace{1mm}

\subsection{Bounds on the renormalization derivative $D \tilde F \varphi_k$} $\phantom{a}$ \\ \label{der5}

\vspace{1mm}

\noindent $1)$ According to the definition $(\ref{tildev})$:
\begin{equation}
\tag{$D \tilde v \varphi_k$} \label{Dvh6} \boxed{ D \tilde v \varphi_k=2{k_1 D k_2 \varphi_k-k_2 D k_1 \varphi_k -v (D k_2 \varphi_k - D k_1 \varphi_k)  \over (k_2-k_1)^2 },}
\end{equation}
and
\begin{align}
  \label{Dk1h6}   D k_1 \varphi_k  & = -(\phi^{-1})'(s_J(\psi^{-1}(-t))) \varphi_k( \phi^{-1}(s_J(\psi^{-1}(-t)))),\\
  \label{Dk2h6}   D k_2 \varphi_k  & = -(\phi^{-1})'(s_J(\psi^{-1}(t))) \varphi_k( \phi^{-1}(s_J(\psi^{-1}(t)))).
\end{align}

\vspace{1mm}

\noindent $2)$ The next two derivtives:
\begin{align}
  \tag{$D \tilde i \varphi_k$}   \label{Dtih6} &\boxed{D \tilde i \varphi_k = -{(\phi^{-1})'(t) e_k(\phi^{-1}(t))  \over d (1-v)^{1 \over d}  (\phi^{-1}(t)-v)^{d-1 \over d} },} \\
\tag{$D \tilde j \varphi_k$} \label{Dtjh6} &\boxed{ D \tilde j \varphi_k = -{(\phi^{-1})'(-t) e_k(\phi^{-1}(-t))  \over d (1-v)^{1 \over d}  (\phi^{-1}(-t) -v)^{d-1 \over d}}. }
\end{align}

\vspace{1mm}

\noindent $3)$ Next derivative:
\begin{equation}
\tag{$D \tilde t \varphi_k$} \label{Dth6} \boxed{D \tilde t \varphi_k= {D k_2 \varphi_k \over d (1-v)^{1 \over d}  (k_2 -v)^{d-1 \over d}} .}
\end{equation}

\vspace{1mm}

\noindent $4)$ Finally,
\begin{align}
\tag{$D \tilde \psi \varphi_k$} \label{DtpsiphiK} \boxed{ D \tilde \psi \varphi_k=-{1 \over \lambda_T} \left( \varphi_K \hspace{-0.8mm}  \co  \hspace{-0.8mm} p_v  \hspace{-0.8mm} \co  \hspace{-0.8mm} s_{\tilde J} +(\phi'  \hspace{-0.8mm} \co  \hspace{-0.8mm} p_v \co  \hspace{-0.8mm} s_{\tilde J}) (p_v'  \hspace{-0.8mm} \co  \hspace{-0.8mm} s_{\tilde J}) D s_{\tilde J} \varphi_k \right),}
\end{align}
where
\begin{equation}
(D s_{\tilde J} \varphi_k)(x)={D \tilde j \varphi_k-D \tilde i \varphi_k  \over 2   } x+{D \tilde j \varphi_k+ D \tilde i \varphi_k \over 2}.  
\end{equation}

Furthermore,
\begin{equation}
\tag{$D \tilde \phi \varphi_k$} \label{DtphiphiK} \boxed{ D \tilde \phi \varphi_k={1 \over \lambda_T \lambda_J}  (\psi' \hspace{-0.8mm} \co  \hspace{-0.8mm} s_J^{-1}  \hspace{-0.8mm} \co  \hspace{-0.8mm} \phi  \hspace{-0.8mm} \co  \hspace{-0.8mm} s_K) \left( (\varphi_k  \hspace{-0.8mm} \co  \hspace{-0.8mm} s_K)+(\phi'  \hspace{-0.8mm} \co  \hspace{-0.8mm} s_K) D s_K \varphi_k \right)},
\end{equation}
where
\begin{equation}
(D s_K \varphi_k)(x)={D k_2 \varphi_k-D k_1 \varphi_k  \over 2   } x+{D k_2 \varphi_k+ D k_1 \varphi_k \over 2}.  
\end{equation}

\vspace{1mm}

\subsection{Interval arithmetics in $\R$, $C$ and $\cD^\omega$}

We will now give a very brief summary of interval operations in $\R$ and $\cD^\omega$. For a more complete treatise of ``standard sets'' and operations on them, an interested reader is referred to an excellent review \cite{KSW96}. 

A computer implementation of an arithmetic operation  $r_1 \# r_2$ ($\#$ is $+$, $-$, $*$ or $/$) on two real numbers does not generally yield an exact result. The ``computer'' result is a number representable in a standard  IEEE floating point format (cf \cite{IEEE}). Such numbers are commonly referred to as ``representable''. In the $64$-bit, or double, precision IEEE arithmetics, used in our proofs, a number is represented with $64$ bits of memory: $1$ bit for the sign of the number, $11$ bits for the exponent $e$, and $52$ bits, $m_0,m_1, \ldots, m_{51}$,  for a mantissa $m$. A real representable number $r$ is of a the form 
$$r= (-1)^{\rm sign} \left( 1+  \sum_{i=1}^{52} m_{52-i} \right) \ 2^{e-1023},$$ 
where the exponent $0 \le e \le 2047$ is defined by the state of the $11$ exponent bits (the state of $11$ ones is reserved to represent  ``overflows''). The set of all representable numbers will be denoted $\cR$.

Now, let $r_1 \# r_2$ be a mathematically legal, nonzero, arithmetic operation. The result of the true arithmetic operation  $r_1 \# r_2$ might not be a representable number. However, we can instruct the computer to attempt to round this operation either to the nearest representable number, up or down (our choice was to always round up). This might not be possible: the result of rounding up might have the exponent $e \ge  2047$ (overflow), or $e <0$ (underflow). In both cases the computer is instructed to raise an exception,  which is appropriately handled (either by terminating the program, or restarting with a different set of parameters). If, however, the result of rounding up is representable, the output of the computer implementation of the arithmetic operation is an upper bound on the true result of the operation. We will refer to such bound as ``the upper bound''.

We will define the {\it standard sets} in $\R$ to be the collection of all closed real intervals $I[x,y]=\{r \in \R: x \le r \le y; x,y \in \cR  \}$:  
\begin{equation}
{\rm std}{(\R)}=\{I[x,y] \in \R: x,y \in \cR \}.
\end{equation}  

If $I[x,y]$, $I[x_1,y_1]$ and $I[x_2,y_2]$ are in ${\rm std}{(\R)}$  then we can use the rounding up described above to obtain bounds on  the arithmetic operation on these sets as follows:

\begin{itemize}
\item[1)] {\it unary minus}:    $-(I[x,y])=I[-y,-x];$
\item[2)] {\it absolute value}: $|I[x,y]|=I[l,r]$, where  $l=\max\{ 0, x, -y \}$, $r=-\min\{0,x, -y\}$;
\item[3)] {\it addition}: $I[x_1,y_1]+ I[x_2,y_2]=I[x_3,y_3]$, where $y_3$ is the upper bound on $y_1+y_2$, while $-x_3$ is the upper bound on $-x_1+(-x_2)$;
\item[4)] {\it subtraction}: $I[x_1,y_1]- I[x_2,y_2] \equiv I[x_1,y_1]+ (-I[x_2,y_2])$;
\item[5)] {\it multiplication}: $I[x_1,y_1] \cdot I[x_2,y_2]=I[x_3,y_3]$, where $y_3$ is the maximum of the upper bounds on  $x_1  \cdot x_2$, $x_1 \cdot y_2$, $y_1  \cdot x_2$ and $y_1 \cdot y_2$, while $-x_3$ is the maximum of the upper bounds on $(-x_1)  \cdot x_2$, $(-x_1) \cdot y_2$, $(-y_1)  \cdot x_2$ and $(-y_1) \cdot y_2$; 
\item[6)] {\it inverse}: if $x_1 \cdot y_1 >  0$, then $I[x_1,y_1] \ {\rm inverse}=I[x_2,y_2]$, where $y_2$  is  the upper bound on $1/x_1$ and $-x_2$  is the  upper bound on $1/(-y_1)$;
\item[7)] {\it division}:  if $x_2 \cdot y_2 >  0$, then $I[x_1,y_1] / I[x_2,y_2] \equiv I[x_1,y_1] \cdot (I[x_2,y_2] \ {\rm inverse})$. 
\end{itemize}

Standard sets in $\C$ have been chosen to be disks in the complex plane of the form
\begin{equation}
{\rm std}{(\C)}=\{\D_{r}((a,b)) \in \C: a \in \cR, b \in \cR, r \in \cR \cap [0,\infty)  \}.
\end{equation}  
Arithmetic operations on ${\rm std}{(\C)}$ are straightforwardly reducible to arithmetic operations on three copies of ${\rm std}{(\R)}$.

Standard sets in $\cD^\omega$ are chosen to be of the form
\begin{equation}
{\rm std}{(\cD^\omega)}=\{ ({\rm std}(\R) \times \cR_+)^{N+1}: \cR_+=\cR \cap \{x \ge 0\},
\end{equation}
that is, $I$ is a standard set in $\R$, and $\cR_+$ is the set of non-negative representable numbers and $N$ is a non-negative integer. A fuction $\psi \in \cD^\omega$ can be represented as
\begin{equation}
\psi=\sum_{k=0}^N a_k x^k+g_k(x),
\end{equation}
where $g_k=O(x^k)$ and is real-analytic on $[-1,1]$, with the norm $\| g_k \|_1 \le  r_k$. Thus, each $(a_k, g_k)$ can be bounded by some $I \times r_k \in  {\rm std}(\R) \times \cR_+$, and therefore, elements of ${\rm std}{(\cD^\omega)}$ serve as bounds on elements of $\cD^\omega$.

The arithmetic operation on these standard sets in $\cD^\omega$ are reducible to those in ${\rm std}(\R)$, and operations on representable numbers.

The standard sets in ${\rm std}{(\cD^\omega)}$ will denoted by $\{D[I_0,r_0], \ldots,D[I_N,r_N]\}$.

\vspace{1mm}

\subsection{Rigorous bound on the renormalization fixed point through a Contraction Mapping Principle}

We will now outline a rather general method for finding a fixed point of a hyperbolic operator in a Banach space via its approximate Newton map. 

Let $\cR$ be an operator analytic and hyperbolic on some neighborhood $\cN$ in a Banach space $\cZ$. Suppose, that one knows its approximate hyperbolic fixed point  $Z_0 \in \cN$. Set
$$M \equiv \left[ \I-D_a \right]^{-1},$$
where $D_a$ is some finite-dimensional approximation of $D \cR[Z_0]$, and define, for all $z$, such that $Z_0+M z \in \cN$,
\begin{equation}
  \label{newtonOP} N[z]=z+\cR[Z_0+M z]-(Z_0+M z).
\end{equation}
Notice, that if $z^*$ is a fixed point of $N$, then $Z_0+M z^*$ is a fixed point of $C$.

The linear operator $\I-D_a$ is indeed invertible since $D \cR$ is hyperbolic at $Z_0$. If $Z_0$ is a reasonably good approximation of the true fixed point of $\cR$, then the operator $N$ is expected to be a strong contraction in a neighborhood of $0$:
\begin{eqnarray}
\nonumber D N[z] &=& \I +D \cR[Z_0+M z] \cdot M -M \\
\nonumber &=& \left[ M^{-1} +D \cR[Z_0+M z]  -\I \right] \cdot M \\
\nonumber &=& \left[\I-D_a  +D \cR[Z_0+M z]  -\I \right] \cdot M \\
\nonumber &=& \left[ D \cR[Z_0+M z] - D_a  \right] \cdot M.
\end{eqnarray}

The last expression is typically small in a small neighborhood of $0$, if the norm of $M$ is not too large (if $M$ is large, one might have to find a better approximation $Z_0$ and take a smaller neighborhood of $0$).

We have used the Contraction Mapping principle to obtain rigorous bounds on the renormalization fixed point in two cases: $d=4.0$ and $d=5.0$.

\begin{exthm} $\phantom{a}\\$

\noindent {\bf Case d=3.8}

\vspace{1mm}

  There exist two polynomials of degree $N=300$, $\psi_{3.8}$ and $\phi_{3.8}$,  such that the Newton operator $(\ref{newtonOP})$ for $\cR$, defined in subsection $\ref{real_ren}$, is a contraction of the ball $B_\delta(Q_a)=\{Q \in \cB: \| Q-Q_a\|_1 < \delta\}$, where $\delta=4.28577601155835506 \times 10^{-12}$, and
  \begin{align}
    \label{ga} Q_a=(&-0.83700583021901265,-0.89842138158302138,\\
    \nonumber &-0.64345222855602410,0.44908966263509253, \psi_{3.8},\phi_{3.8}),
    \end{align}
  with
  \begin{align*}
    \eps&=\|\cN(Q_a)-Q_a\|_1  \le 3.89968810858366172 \times 10^{-13} \\
    D&=\|D N \arrowvert_{B_\delta}  \|_1 \le 8.35798206185280828 \times 10^{-7}. 
    \end{align*}
In particular, $B_\delta(Q_a)$ contains a fixed point of the operator $N$.

\vspace{1mm}

\noindent {\bf Case d=5.1}

\vspace{1mm}

  There exist two polynomials of degree $N=300$, $\psi_{5.1}$ and $\phi_{5.1}$,   such that the Newton operator $(\ref{newtonOP})$ for $\cR$, defined in subsection $\ref{real_ren}$, is a contraction of the ball $B_\delta(H_a)=\{H \in \cB: \| H-H_a\|_1 < \delta\}$, where $\delta=3.34356454892135768 \times 10^{-12}$, and
  \begin{align}
    \label{ha} H_a=(&-0.89271858672458115,-0.93765644691994054,\\
    \nonumber &-0.68211126486048014,0.56526098770489580, \psi_{5.1},\phi_{5.1}),
    \end{align}
  with
  \begin{align*}
    \eps&=\|\cN(H_a)-H_a\|_1  \le 3.41452789569873455 \times 10^{-13} \\
    D&=\|D N \arrowvert_{B_\delta}  \|_1 \le 7.13468923563748925 \times 10^{-7}. 
    \end{align*}
In particular, $B_\delta(H_a)$ contains a fixed point of the operator $N$.

\end{exthm}

\begin{proof}
  
\noindent   $1)$ A computer implementation of the renormalization operator reduces to implementations of compositions and algebraic operations on ${\rm std}(\cR)$ and ${\rm std}(\cD^\omega)$.

  In particular, $\eps$ is estimated via an implementation of $\cR$ on a standard set of the form $\{D[I_1[a_1,a_1],0], D[I_2[a_2,a_2],0], \ldots, D[I_N[a_N,a_N],0]\}$, where $a_k \in \cR$ and $N=300$.

\noindent   $2)$ The derivative $D N \arrowvert_{B_\delta}$ has been rigorously estimaed on the computer via a implementation of formulas in subsections $(\ref{der1})-(\ref{der5})$ over  standard sets in $\cB$. In particular, the derivatives $D \tilde v \eta_k$, $D \tilde i \eta_k$, $D \tilde j \eta_k$, $D \tilde t \eta_k$, $D \tilde \psi \eta_k$, $D \tilde \phi \eta_k$ and $D \tilde v \varphi_k$, $D \tilde i \varphi_k$, $D \tilde j \varphi_k$, $D \tilde t \varphi_k$, $D \tilde \psi \varphi_k$, $D \tilde \phi \varphi_k$ have neen computed for $k=0, \ldots, 250$.

  The derivative $D \tilde v \eta$ with $\eta \in \cT$, $\| \eta\|_1=1$, where $\cT$ is defined by $T \cD^\omega={\rm span}\{e_0,  \ldots, e_n\} \bigoplus \cT$, $n=250$, are estimated via an implementaton of (\ref{Dvh5})  on the standard set of the form
  $$\{\underbrace{D[0,0], \ldots, D[0,0]}_{n},D[0,1]\},$$
$n=250$. Actions of all other derivatives on higher-order terms are estimated in a similar way.
\end{proof}

\section{Wild attractors for Fibonacci maps} \label{attractors}

Throughout this Section we will use the symbol $Q$ for the fixed point $F$ in he case $d=3.8$, and $H$ for the fixed point $F$ in the case $d=5.1$.

We start by providing a computer estimate on the Schwarzian and the Koebe constant
\begin{lem}\label{KoebeConst} $\phantom{a}\\$

\noindent $1)$  The Schwarzian for the renormalization fixed point $Q$ is non-positive, and vanishes only at $x=0$.
  
  The Koebe constant from Corollary $\ref{CoKoebeDist}$ satisfies
\begin{equation}
\label{constC4}  C < 13.4664644314052974.
\end{equation}

\vspace{1mm}

\noindent $2)$   The Schwarzian for the renormalization fixed point $H$ is non-positive, and vanishes only at $x=0$.
  
  The Koebe constant from Corollary $\ref{CoKoebeDist}$ satisfies
\begin{equation}
\label{constC5}  C < 29.4431036985348317.
\end{equation}
\end{lem}
\begin{proof}
  The non-positivity of the Schwarzian is verified with interval arithmetics by brut force by checking sign of the relation
  \begin{equation}
   { F'''(x) F'(x) -{3 \over 2} F''(x) \over    |x|^{2 d -4} },
  \end{equation}
with $F=G$, $d=3.1$ or $F=H$, $d=5.1$,    over a fine interval partition of the domain $J_1 \cup T_1$.
     By the classical Koebe distortion principle,
     \begin{equation}
        \label{constC1}       C < { (1+\tau)^2 \over \tau^2},
      \end{equation}
whenever the Schwarzian is negative, where $\tau$ is the Koebe space of the interval $T_1$ inside $[x_4, -\lambda^{-1} x_4]$.
\end{proof}

\begin{rem}
The programs for the computation of $\zeta_n$ and $\eta_n$ in the following two subsections have been parallelized and run on $28$ Intel(R) Xeon(R) CPU E5-2620 runnig at 2.00GHz with 6 cores each. 
\end{rem}

\subsection{Non-existence of wild attractors for $d = 3.8$}

The specific value $(\ref{constC4})$ of the Koebe constant is used in the following

\begin{thm}\label{nonexistence}
The Fibonacci renormalization cycle does not admit a wild attractor for $d=3.8$.
\end{thm}
\begin{proof}
  The relation $\eta_n/\zeta_n < C^{-1}$ has been verified for the fixed point $Q$. Specifically, for $n=9$,
  \begin{eqnarray}
    \label{etan}    \eta_n  &<&0.034670807835335466, \\
    \label{zetan}   \zeta_n &>&0.620936555002023912. 
  \end{eqnarray}

  \bigskip
  
  \noindent $1)$ {\it Estimating $\zeta_n$ from below (Algorithm $(\ref{alg1})$).} An interval $I$ from a partition of $T_1$ is iterated until time $Z=50000$. We check if $Q^k(I) \not \subset T_1 \cup J_1$, for some $k \le Z$. If this holds, the interval is recorded. Otherwise, if  $Q^k(I) \cap \partial (T_1 \cup J_1) \neq \emptyset$, the interval $I$ is partitioned in two subintervals of equal length, and the same is checked for the two subintervals. Subdivision is performed no more than a finite fixed number of times.

Set $I_n=\lambda^{n-1} I \subset T_n$. If $K=Q^k(I) \not \subset T_1 \cup J_1$ indeed, for some $k \le N$, then $K_n=\lambda^{n-1} K =Q_{n-1}^k(I_n) \not \subset T_n \cup J_n$.

Such interval $K_n$ is iterated upto $X=50000$: if for some $l \le X$, $Q^l(K_n) \notin T_1 \cup J_1$ and $Q^j(K_n) \notin T_n$ and $T_n \notin Q^j(K_n)$  for all $0 \le j <l$, then $I_n$ is counted in the lower bound of the zet $Z_n$.

If  for some $l \le L$, $Q^l(K_n) \cap \partial (T_1 \cup J_1) \neq \emptyset$ or  $Q^j(K_n) \cap \partial T_n \neq \emptyset$, then the interval $I$ is subdivided and the whole procedure is restarted. The subdivision of an interval $I$ may happen only a fixed finite number of times. If, after this finite number of times, one of the smaller subintervals (also referred to as  $I$) still satisfies $Q^l(K_n) \cap \partial (T_1 \cup J_1) \neq \emptyset$ or  $Q^j(K_n) \cap \partial T_n \neq \emptyset$, then such $I$ is not counted in the lower bound on $Z_n$.

The algorithm described above is implemented in {\tt{zeta\_below}}.

\vspace{1mm}

\noindent $2)$ {\it Estimating $\eta_n$ from above.} All intervals $I$ in a partition of $T_1$, such that $Q^k(I) \subset T_1 \cup J_1$ for all $k \le N$, $N=50000$  are counted in an upper bound on the set $X_n$. 
\end{proof}

\subsection{Existence of wild attractors for $d = 5.1$}

The specific value $(\ref{constC5})$ of the Koebe constant is used in the following

\begin{thm}\label{existence}
The Fibonacci renormalizations cycle admits a wild attractor for $d = 5.1$.
\end{thm}
\begin{proof}
  The relation $\eta_n/\zeta_n > C$ has been verified for the fixed point $H$. Specifically, for $n=4$,
  \begin{eqnarray}
    \label{etan}    \eta_n  &>&0.52383224533154263, \\
    \label{zetan}   \zeta_n &<&0.01634562784562387. 
  \end{eqnarray}

  \bigskip
  
  \noindent $1)$ {\it Estimating $\zeta_n$ from above (Algorithm $(\ref{alg2})$).} First, we take preimages of the interval $T_n$ down to level $N=18$ (routines {\tt{preimages\_zeta}} and {\tt{preimages\_zeta\_next}}). 

  Next, preimages under $Q_{n-1}$ are taken downto level $K=11$. All connected components of the intersection of these preimages with $T_n$ are counted in a lower bound on the set $T_n\setminus Z_n$.
  
\vspace{1mm}

\noindent $2)$ {\it Estimating $\eta_n$ from below.}  All intervals $I$ in a partition of $T_1$, such that $H^k(I) \subset T_n$ for some $k \le N$, $N=50000$, is counted in a lower bound on the set $X_n$.

\end{proof}

\begin{rem}
  Since the renormalization fixed point is locally continous with respect to $d$, we also obtain that non-existence of wild attr
  actors for $d$ in a neighborhood of $d=3.8$, and existence in a neighborhood of $d=5.1$. 
\end{rem}

\section{Appendix: Description of the programs}

The programs used for he rigorous estimates have been written in the programming language ADA.

\subsection{Upper level shell-scripts and executables} $\phantom{a}\\$

\noindent {\bf Script {\path{run_eta}}.} \newline
The shell script {\path{run_eta}} reads the file {\path{hosts}}, checks if each host from this list is ssh-accesible, and writes the accesible ones into the file {\path{work_hosts}}. After that it acceses every hosts from {\path{work_hosts}} via ssh and starts there, in the background,  $r$ copies of the executable {\path{eta}}, closing the ssh connection after that. $i$-th copy of {\path{eta}}  outputs into {\path{output_eta.D.M.(L+1).i}}. The parameters supplied to the executable {\path{eta}} are
\begin{itemize}
\item[$\bullet$] $D$, the degree of the power map;
\item[$\bullet$] $L$, $L+1$ is the renormalization level $n$ in $J_n \cup T_n$;
\item[$\bullet$] $M$, defines the size of the intervals $I$ in the partition of $T_1$ used to find $\eta_n$, specifically $|I|=|T_n|/M$; 
\item[$\bullet$] $N$, maximum iteration time of each interval $I$  (see Part $2)$ of Theorem $\ref{nonexistence}$);
\item[$\bullet$] $F$, the refinement size cutoff for an interval $I$ is  $F \cdot |I|$  (see Part $2)$ of Theorem $\ref{nonexistence}$);
\item[$\bullet$] $m$, total number of executed copies  $m=r \cdot \!$  \#\{{\path{work_hosts}}\};
\item[$\bullet$] $i$, $i$-th copy of {\path{eta}} is executed. The $i$-th copy iterates all intervals $I_j$ with $(i-1) \lfloor p/m \rfloor+1 \le j \le i \lfloor p/m \rfloor$,  where $p$ is the total number of executed copies and $p$ is the cardinality of the partion $I_j$. 
\end{itemize}

\vspace{2mm}

\noindent {\bf Script {\path{run_zeta_below}}.} \newline
The shell script {\path{run_zeta_below}} reads the file {\path{hosts}}, checks if each host from this list is ssh-accesible, and writes the accesible ones into the file {\path{work_hosts}}. After that it acceses every hosts from {\path{work_hosts}} via ssh and starts there, in the background,  $r$ copies of the executable {\path{zeta_below}}, closing the ssh connection after that. $i$-th copy of {\path{zeta_below}}  outputs into {\path{output_zeta_below.D.M.(L+1).i}}. The parameters supplied to the executable {\path{zeta_below}} are
\begin{itemize}
\item[$\bullet$] $D$, the degree of the power map;
\item[$\bullet$] $L$, $L+1$ is the renormalization level $n$ in $J_n \cup T_n$;
\item[$\bullet$] $M$, defines the size of the intervals $I$ in the partition of $T_1$ used to find $\eta_n$, specifically $|I|=|T_n|/M$; 
\item[$\bullet$] $Z$, maximum iteration time $Z$ of each interval $I$  (see Part $1)$ of Theorem $\ref{nonexistence}$);
\item[$\bullet$] $X$, maximum iteration time $X$ of each interval $K_n$  (see Part $1)$ of Theorem $\ref{nonexistence}$);
\item[$\bullet$] $F1$, used in the verification  of $G^k(I) \cap \partial (T_1 \cup J_1) \neq \emptyset$ (see Part $1)$ of Theorem $\ref{nonexistence}$), the refinement size cutoff for an interval $I$ is  $F1 \cdot |I|$;
\item[$\bullet$] $F2$, used in the verification  of $F^l(K_n) \cap \partial (T_1 \cup J_1) \neq \emptyset$ or  $F^j(K_n) \cap \partial T_n \neq \emptyset$  (see Part $1)$ of Theorem $\ref{nonexistence}$),  the refinement size cutoff for an interval $I$ is  $F2 \cdot |I|$;
\item[$\bullet$] $m$, total number of executed copies  $m=r \cdot \!$  \#\{{\path{work_hosts}}\};
\item[$\bullet$] $i$, $i$-th copy of {\path{eta}} is executed. The $i$-th copy iterates all intervals $I_j$ with $(i-1) \lfloor p/m \rfloor+1 \le j \le i \lfloor p/m \rfloor$,  where $p$ is the total number of executed copies and $p$ is the cardinality of the partion $I_j$. 
\end{itemize}

\vspace{2mm}

\noindent {\bf Executable {\path{preimages_zeta}}.} \newline
The executable  {\path{preimages_zeta}} computes preimages of $T_n$ upto certain depth level. The input is the renormalization data from the files {$\sim$\path{/Data/psi_d}}, {$\sim$\path{/Data/phi_d}} and {$\sim$\path{/Data/data_d}}.    The parameters supplied are
\begin{itemize}
\item[$\bullet$] $d$, the degree of the power map;
\item[$\bullet$] $l$, $l+1$ is the renormalization level $n$ in $J_n \cup T_n$;
\item[$\bullet$] $n$, depth of preimages of $T_n$ under $F$. 
\end{itemize}
The left and the right endpoints of the preimages are output into {$\sim$\path{/Data/leftpoints_zeta_d_(l+1)_0}} and  {$\sim$\path{/Data/rightpoints_zeta_d_(l+1)_0}}, respectively.

This executable also outputs the indices $i$ of the first and last interval in the deepest level, to the current output. 

\vspace{2mm}

\noindent {\bf Script {\path{run_preimages_zeta_next}}.} \newline
The shell script {\path{run_preimages_zeta_next}} reads the file {\path{hosts}}, checks if each host from this list is ssh-accesible, and writes the accesible ones into the file {\path{work_hosts}}. After that it acceses every hosts from {\path{work_hosts}} via ssh and starts there, in the background, $M$ copies of the executable {\path{preimages_zeta_next}}, closing the ssh connection after that. $i$-th copy of {\path{preimages_zeta_next}}  outputs into {\path{output_preimages_zeta_next.D.M.(L+1).i}}. $M$ is the number of intervals of the deepest level in the output of {\path{preimages_zeta}}, their endpoints having been saved in  {$\sim$\path{/Data/leftpoints_zeta_D_(L+1)_0}} and  $\sim${\path{Data/rightpoints_zeta_D_(L+1)_0}}. Thus, every copy of  {\path{preimages_zeta_next}} computes preimages of a single interval from the deepest level of the output of  {\path{preimages_zeta}}.

The parameters supplied to the executable {\path{preimages_zeta_next}} are
\begin{itemize}
\item[$\bullet$] $D$, the degree of the power map;
\item[$\bullet$] $L$, $L+1$ is the renormalization level $n$ in $J_n \cup T_n$;
\item[$\bullet$] $Q$, index of the first interval in the deepest level in the output of {\path{preimages_zeta}};
\item[$\bullet$] $P$, index of the last interval in the deepest level in the output of {\path{preimages_zeta}};
\item[$\bullet$] $M$, number of the intervals in the deepest level in the output of {\path{preimages_zeta}};
\item[$\bullet$] $K$, level of preimages in {\path{preimages_zeta_next}}.
\end{itemize}

\vspace{2mm}

\noindent {\bf Executable {\path{preimages_zeta_next}}.} \newline
The executable  {\path{preimages_zeta_next}} computes preimages of an interval of the deepest level in the output of  {\path{preimages_zeta}},  upto certain depth level. The input is the renormalization data from the files {$\sim$\path{/Data/psi_D}}, {$\sim$\path{/Data/phi_D}}, {$\sim$\path{/Data/data_D}} and a pair of left and right endpoints from {$\sim$\path{/Data/leftpoints_zeta_D_(L+1)_0}} and  {$\sim$\path{/Data/rightpoints_zeta_D_(L+1)_0}}.

The parameters supplied to this executable have been described under {\bf Script {\path{run_preimages_zeta_next}}}

This executable outputs left and right endpoints of the preimages of the $i$-th interval from the deepest level of the output of {\path{preimages_zeta}}  into {$\sim$\path{/Data/leftpoints_zeta_D_(L+1)_i}} and  {$\sim$\path{/Data/rightpoints_zeta_D_(L+1)_i}}.

\vspace{2mm}

\noindent {\bf Executable {\path{reduce_intervals}}.} \newline
The executable  {\path{reduce_intervals}} finds the union of all intervals with left and right endpoints from   {$\sim$\path{/Data/leftpoints_zeta_d_(l+1)_i}} and  {$\sim$\path{/Data/rightpoints_zeta_d_(l+1)_i}}. The input is the renormalization data from the files {$\sim$\path{/Data/psi_d}}, {$\sim$\path{/Data/phi_d}}, {$\sim$\path{/Data/data_d}} and pairs of left and right endpoints from {$\sim$\path{/Data/leftpoints_zeta_d_(l+1)_i}} and  {$\sim$\path{/Data/rightpoints_zeta_d_(l+1)_i}}. 
The parameters supplied to the executable {\path{reduce_intervals}} are
\begin{itemize}
\item[$\bullet$] $d$, the degree of the power map;
\item[$\bullet$] $l$, $l+1$ is the renormalization level $n$ in $J_n \cup T_n$;
\item[$\bullet$] $o$, number of files {$\sim$\path{/Data/leftpoints_zeta_d_(l+1)_i}}, or, equivalently, {$\sim$\path{/Data/rightpoints_zeta_d_(l+1)_i}}.
\end{itemize}

This executable outputs left and right endpoints of the intervals after taking into {$\sim$\path{/Data/lleftpoints_zeta_d_(l+1)_i}} and  {$\sim$\path{/Data/rrightpoints_zeta_d_(l+1)_i}}. 

\vspace{2mm}

\noindent {\bf Script {\path{run_ren_preimages_zeta_next}}.} \newline
The shell script {\path{run_ren_preimages_zeta_next}} reads the file {\path{hosts}}, checks if each host from this list is ssh-accesible, and writes the accesible ones into the file {\path{work_hosts}}. After that it acceses every hosts from {\path{work_hosts}} via ssh and starts there, in the background, $M+1$ copies of the executable {\path{preimages_ren_zeta_next}}, closing the ssh connection after that. $i$-th copy of {\path{preimages_ren_zeta_next}}  outputs into {\path{output_ren_preimages_zeta_next.D.M.(L+1).i}}. $M$ is the number of files {$\sim$\path{/Data/lleftpoints_zeta_D_(L+1)_i}}, or, equivalently,  {$\sim$\path{/Data/rrightpoints_zeta_D_(L+1)_i}}.

Every copy of  {\path{preimages_ren_zeta_next}} computes preimages under $F_{n-1}$ of every preimage of  $T_n$ under $F$.

The parameters supplied to the executable {\path{preimages_zeta_next}} are
\begin{itemize}
\item[$\bullet$] $D$, the degree of the power map;
\item[$\bullet$] $L$, $L+1$ is the renormalization level $n$ in $J_n \cup T_n$;
\item[$\bullet$] $M$, number of the intervals in the deepest level in the output of {\path{preimages_zeta}};
\item[$\bullet$] $K$, deepest level of preimages under $F_{n-1}$  in {\path{preimages_ren_zeta_next}}.
\end{itemize}

\vspace{2mm}

\noindent {\bf Executable {\path{preimages_ren_zeta_next}}.} \newline
The executable  {\path{preimages_ren_zeta_next}} computes preimages under $F_{n-1}$  of all intervals with left and right endpoints from   {$\sim$\path{/Data/lleftpoints_zeta_D_(L+1)_i}} and  {$\sim$\path{/Data/rrightpoints_zeta_D_(L+1)_i}}, upto certain depth level. The input is the renormalization data from the files {$\sim$\path{/Data/psi_D}}, {$\sim$\path{/Data/phi_D}}, {$\sim$\path{/Data/data_D}} and pairs of left and right endpoints from {$\sim$\path{/Data/lleftpoints_zeta_D_(L+1)_i}} and  {$\sim$\path{/Data/rrightpoints_zeta_D_(L+1)_i}}.

The parameters supplied to this executable have been described under {\bf Script {\path{run_ren_preimages_zeta_next}}}.

This executable outputs left and right endpoints of the intersections preimages under $F_{n-1}$ with $T_n$  into {$\sim$\path{/Data/renorm_leftpoints_zeta_D_(L+1)_i}} and  {$\sim$\path{/Data/renorm_rightpoints_zeta_D_(L+1)_i}}.

\vspace{2mm}

\noindent {\bf Executable {\path{compute_zeta}}.} \newline
The executable  {\path{compute_zeta}} finds the union of all intervals whose end points are saved in {$\sim$\path{/Data/renorm_leftpoints_zeta_d_(l+1)_i}} and  {$\sim$\path{/Data/renorm_rightpoints_zeta_d_(l+1)_i}}, and computes $\zeta_n$.

The parameters supplied to the executable {\path{preimages_zeta_next}} are
\begin{itemize}
\item[$\bullet$] $d$, the degree of the power map;
\item[$\bullet$] $l$, $l+1$ is the renormalization level $n$ in $J_n \cup T_n$;
\item[$\bullet$] $o$, number of files {$\sim$\path{/Data/renorm_leftpoints_zeta_d_(l+1)_i}}, or equivalently,  {$\sim$\path{/Data/renorm_rightpoints_zeta_d_(l+1)_i}}, and computes $\zeta_n$;
\item[$\bullet$] $m$, deepest level of preimages under $F_{n-1}$  in {\path{preimages_ren_zeta_next}}.
\end{itemize}

\section{Open problems}
Since Conjecture \ref{conj_existence} remains open, it is not known yet whether any part of Theorem \ref{thm_Aliva-Lyubich} holds true unconditionally. In particular, the following questions remain open.
\begin{problem} Does there exist a Fibonacci map $f$ such that $\mathcal B(\overline{\mathcal P(f)})$ has $a)$ Hausdorff dimension less than 1, $b)$ Hausdorff dimension 1 but zero Lebesgue measure?
\end{problem}
\noindent An approach to show a positive answer to the previous problem is to clarify the relation between the behavior of $\eta_n,\zeta_n$ and Hausdorff dimension of the basin of attraction $B(\mathcal C)$ of periodic point of Fibonacci renormalization.
\begin{problem} Is it true that conditions of part $(1)$ of Theorem A imply that $B(\mathcal C)$ has Hausdorff dimension less than $1$? Do conditions of part $(2)$ imply that $B(\mathcal C)$ has Hausdorff dimension 1?
\end{problem} 
\bibliography{biblio}

\providecommand{\bysame}{\leavevmode\hbox to3em{\hrulefill}\thinspace}
\providecommand{\MR}{\relax\ifhmode\unskip\space\fi MR }
\providecommand{\MRhref}[2]{%
  \href{http://www.ams.org/mathscinet-getitem?mr=#1}{#2}
}
\providecommand{\href}[2]{#2}
\begin{thebibliography}{BKNvS96}

\bibitem[AL08]{Avila_Lyubich_08}
A.~Avila and M.~Lyubich, \emph{Hausdorff dimension and conformal measures of
  {F}eigenbaum {J}ulia sets}, J. Amer. Math. Soc. \textbf{21} (2008), no.~2,
  305--363. \MR{2373353}

\bibitem[AL22]{AvilaLyubich-Area-22}
\bysame, \emph{Lebesgue measure of {F}eigenbaum {J}ulia sets}, Ann. of Math.
  (2) \textbf{195} (2022), no.~1, 1--88. \MR{4358413}

\bibitem[BKNvS96]{BKNS}
H.~Bruin, G.~Keller, T.~Nowicki, and S.~van Strien, \emph{Wild {C}antor
  attractors exist}, Ann. of Math. (2) \textbf{143} (1996), no.~1, 97--130.
  \MR{1370759}

\bibitem[Bru98]{Bruin-98}
Henk Bruin, \emph{Topological conditions for the existence of absorbing
  {C}antor sets}, Trans. Amer. Math. Soc. \textbf{350} (1998), no.~6,
  2229--2263. \MR{1458316}

\bibitem[BT15]{BruinTodd-15}
Henk Bruin and Mike Todd, \emph{Wild attractors and thermodynamic formalism},
  Monatsh. Math. \textbf{178} (2015), no.~1, 39--83. \MR{3384890}

\bibitem[Buf00]{Buff-00}
Xavier Buff, \emph{Fibonacci fixed point of renormalization}, Ergodic Theory
  and Dynamical Systems \textbf{20} (2000), no.~5, 1287–1317.

\bibitem[DS20]{DS-20}
Artem Dudko and Scott Sutherland, \emph{On the {L}ebesgue measure of the
  {F}eigenbaum {J}ulia set}, Invent. Math. \textbf{221} (2020), no.~1,
  167--202. \MR{4105087}

\bibitem[Dud20]{Dudko_20}
Artem Dudko, \emph{On {L}ebesgue measure and {H}ausdorff dimension of {J}ulia
  sets of real periodic points of renormalization}, Bull. Pol. Acad. Sci. Math.
  \textbf{68} (2020), no.~2, 151--168. \MR{4246926}

\bibitem[Guc79]{Guck-1979}
John Guckenheimer, \emph{Sensitive dependence to initial conditions for
  one-dimensional maps}, Comm. Math. Phys. \textbf{70} (1979), no.~2, 133--160.
  \MR{553966}

\bibitem[IEE08]{IEEE}
\emph{754-2008 - {I}{E}{E}{E} standard for floating-point arithmetic,
  superseded by {I}{E}{E}{E} std 754-2019, a revision of {I}{E}{E}{E}
  754-2008.}, {I}{E}{E}{E} (2008).

\bibitem[KSW96]{KSW96}
H.~Koch, A.~Schenkel, and P.~Wittwer, \emph{Computer-assisted proofs in
  analysis and programming in logic: A case study}, SIAM Review \textbf{38}
  (1996), no.~4, 565--604.

\bibitem[LM93]{LM}
M.~Lyubich and J.~Milnor, \emph{The {F}iboncci unimodal map}, JAMS \textbf{6}
  (1993), no.~2, 425--457.

\bibitem[Lyu95]{Lyubich-ICM-95}
M.~Lyubich, \emph{On the borderline of real and complex dynamics}, Proc. ICM,
  Z\"urich 1994 \textbf{2} (1995), 1203--1215.

\bibitem[NvS94]{nowicki1994polynomial}
Tomasz Nowicki and Sebastian van Strien, \emph{Polynomial maps with a julia set
  of positive measure}, 1994.

\bibitem[Sma07]{Sma}
D.~Smania, \emph{Puzzle geometry and rigidity: the {F}ibonacci cycle is
  hyperbolic}, JAMS \textbf{20} (2007), no.~3, 629--673.

\end{thebibliography}
\bibliographystyle{amsalpha}

\end{document}